\documentclass[11pt]{amsart}
\usepackage[utf8]{inputenc}
\usepackage{amssymb}
\usepackage{microtype}
\usepackage{booktabs}
\usepackage{placeins}
\usepackage{xurl}
\usepackage[hidelinks]{hyperref}

\numberwithin{equation}{section}
\newtheorem{theorem}{Theorem}[section]
\newtheorem{lemma}[theorem]{Lemma}
\newtheorem{proposition}[theorem]{Proposition}
\newtheorem{corollary}[theorem]{Corollary}
\newtheorem{conjecture}[theorem]{Conjecture}
\theoremstyle{definition}

\newcommand{\leg}[2]{\left(\dfrac{#1}{#2}\right)}
\newcommand{\Nm}{\mathrm N}

\title[The Factorial Separation Conjecture]{The Factorial Separation Conjecture}
\author[The Factorial Separation Conjecture]{Valmir Krasniqi}
\address{Institute of Science, Technology, Engineering and Mathematics (STEM),
Rr. Nekibe Kelmendi Nr. 26, 10000 Prishtin\"e, Republic of Kosovo}
\email{valmir@stem-ks.org}
\date{}

\hypersetup{
  pdftitle={The Factorial Separation Conjecture},
  pdfauthor={Valmir Krasniqi},
  pdfsubject={The Factorial Separation Conjecture and unconditional and GRH-conditional results},
  pdfkeywords={Factorial Separation Conjecture, FSC, factorial, perfect power, cubic reciprocity, Hecke character, smooth ideal, explicit formula, GRH}}

\begin{document}
\begin{abstract}
We propose the Factorial Separation Conjecture (FSC), a strict strengthening of earlier factorial gap conjectures:
\[
|n!-m^k|>n
\]
for all integers $n\ge8$ and $m,k\ge2$.
The inequality excludes every displacement up to and including the factorial index.
We prove the conjecture unconditionally for exponents divisible by three, for even exponents at least four, and for every exponent at least two when the largest prime divisor of the base is at most $n+2$.
For exponents divisible by five, we prove the finite interval $8\le n\le1100$ and all sufficiently large indices, with an effective but uncomputed threshold.
The conclusion for every $n\ge8$ and every exponent divisible by five is proved under the Generalized Riemann Hypothesis (GRH) for a specified family of Dedekind zeta functions.
Here GRH asserts that their nontrivial zeros have real part $1/2$.
The arguments combine valuations, cyclotomic factorizations, cubic reciprocity, character estimates and explicit formulas.
All finite comparisons are justified in the text.
The general conjecture remains open.
\end{abstract}
\maketitle

\section{Introduction and main results}\label{sec:introduction}

Small additive perturbations of factorials connect classical factorial equations with power-residue conditions.
We propose the \emph{Factorial Separation Conjecture} (FSC), a strict strengthening of earlier factorial gap conjectures.
A nontrivial perfect power is an integer $m^k$ with integers $m,k\ge2$.
For $n\ge2$, define
\[
\Delta(n)=\min_{m,k\ge2}|n!-m^k|.
\]
The proposed lower bound is $\Delta(n)\ge n+1$ from index eight onward.

\begin{conjecture}\label{conj:fsc}
For every integer $n\ge8$ and all integers $m,k\ge2$,
\begin{equation}\label{eq:fsc}
|n!-m^k|>n.
\end{equation}
\end{conjecture}

The conjecture would imply the Brocard--Ramanujan conjecture that the equation
\begin{equation}\label{eq:brocard}
n!+1=m^2
\end{equation}
has only the positive integer solutions $(n,m)=(4,5),(5,11),(7,71)$; see \cite{BerndtGalway}.
It would also give, for every fixed integer $a$, the uniform cutoff
\[
n<\max\{8,|a|\}
\]
for solutions of $n!+a=m^k$ with $m,k\ge2$.
Thus the excluded displacement grows with the factorial index, uniformly in the base and exponent.
The starting index eight is necessary because $7!+1=71^2$.

The sequences A240939 and A240940 in the On-Line Encyclopedia of Integer Sequences \cite{Orr} record the least shifts above and below a factorial that produce a perfect power.
An online discussion initiated in January 2020 \cite{Mathphile} includes the conjecture that a displacement greater than three producing a perfect power must be at least the factorial index.
For $n\ge8$, Conjecture~\ref{conj:fsc} also excludes the boundary displacement and the shifts $1,2,3$.
The quadratic reductions below build on the reciprocity and factorization arguments in the answers cited in \cite{Mathphile}.

For $m\ge2$, write $P^+(m)$ for its largest prime divisor.
The following table collects the main unconditional results and their proofs.

\begin{theorem}\label{thm:main-unconditional}
Let $n,m,k$ be integers with $m,k\ge2$.
Then $|n!-m^k|>n$ under the assumptions in any row below.
\begin{center}
\normalfont\small
\begin{tabular}{@{}lll@{}}
\toprule
Condition & Range of $n$ & Proof\\
\midrule
$3\mid k$ & $n\ge6$ & \ref{thm:cubic-separation}\\
$k\ge4$ even & $n\ge8$ & \ref{thm:quartic-complete}, \ref{cor:large-even}\\
$P^+(m)\le n+2$ & $n\ge8$ & \ref{thm:bounded-base}\\
$5\mid k$ & $8\le n\le1100$ or $n\ge N_5$ & \ref{cor:quintic-finite}, \ref{thm:quintic-eventual}\\
\bottomrule
\end{tabular}
\end{center}
Here $N_5\ge1101$ is an effectively computable integer independent of $m$ and $k$.
\end{theorem}

The \emph{Generalized Riemann Hypothesis} (GRH) for a Dedekind zeta function asserts that every nontrivial zero, that is, every zero in $0<\operatorname{Re}s<1$, has real part $1/2$.
For a prime $p\ne5$, put $L_p=\mathbb Q(\zeta_5,p^{1/5})$, where $\zeta_5$ is a primitive fifth root of unity.
Our \emph{fifth-power GRH hypothesis} is this assertion for $\zeta_{L_p}$ for every prime $p\ge1100$ satisfying $p^4\equiv1\pmod{25}$.
The next result is conditional on this unproved hypothesis.

\begin{theorem}\label{thm:main-quintic}
Under the preceding fifth-power GRH hypothesis,
\[
|n!-m^k|>n
\]
for every integer $n\ge8$ and all integers $m,k\ge2$ with $5\mid k$.
\end{theorem}

Theorem~\ref{thm:quintic-erh} and Corollary~\ref{cor:quintic-multiples} prove this assertion.
The value of $N_5$ is not computed, so no overlap between the unconditional finite and eventual ranges is established.
Complete fifth-power separation therefore remains conditional.
The general conjecture remains open; Corollary~\ref{cor:final-reduction} describes the residual equations, including the unresolved square case.

\subsection{Relation to earlier work and contributions}\label{sec:contributions}

\emph{Established inputs.}
The progression-product estimate in Lemma~\ref{lem:EO} is formula~(Va) of Erd\H{o}s and Obl\'ath \cite[p.~248]{ErdosOblath}.
Lemma~\ref{lem:planar-lattice} is the sharp planar Hermite bound \cite{Neumaier}; cubic reciprocity is used in the classical form \eqref{eq:cubic-reciprocity} \cite{IrelandRosen}.
Other standard inputs are the global Chebyshev estimate of Rosser and Schoenfeld \cite[Theorem~12]{RosserSchoenfeld}, the ideal-counting, Mertens and Hecke convexity estimates collected in Lemma~\ref{lem:quintic-analytic-inputs} \cite{GarciaLee,Zaman}, and the conductor--discriminant formula \cite{Milne}.

\emph{Adaptations of earlier arguments.}
Theorem~\ref{thm:classification} organizes the quadratic arguments in \cite{Mathphile} into a common parametrization of every possible square displacement.
For odd prime exponents, Theorem~\ref{thm:odd-shift} extends the classical coprime factorial equations treated in \cite{ErdosOblath} to
\[
n!=a^\ell(b^\ell\pm1),\qquad a^\ell\le n.
\]
The added factor $a^\ell$ is controlled uniformly by $n!/a^\ell\ge(n-1)!$.
For fifth powers, Lemma~\ref{lem:quintic-erh} specializes the explicit formulas of Wesolowski \cite[Section~3]{Wesolowski} over $\mathbb Q(\zeta_5)$.
The exact conductor $p\mathcal O_K$ from Lemma~\ref{lem:quintic-exact-conductor} gives the comparison needed for all relevant primes $p\ge1100$.

\emph{Principal contributions.}
The results for the strict bound fall into three groups.
\begin{enumerate}
\item \emph{Elementary and structural exclusions.}
Theorem~\ref{thm:bounded-base} uses Wilson's theorem and $2$-adic lifting of exponents to prove separation under $P^+(m)\le n+2$ for every exponent $k\ge2$.
Consequently every counterexample, including a square counterexample, must have $P^+(m)>n+2$.
At odd prime exponents, the valuation dichotomy and the scaled exclusion leave only prime displacements $n/2<p\le n$.

\item \emph{Complete cubic and even-exponent separation.}
Lemma~\ref{lem:cubic-nonresidue} shows that each prime $q>3$ with $q\equiv\pm1\pmod9$ admits a prime $r<q$, $r\equiv1\pmod3$, for which $q$ is not a cube modulo $r$.
The proof combines cubic reciprocity with the planar lattice bound and yields Theorem~\ref{thm:cubic-separation}.
For fourth powers, put
\[
T_1(n)=\prod_{\substack{q\le n\text{ prime}\\q\equiv1\pmod4}}q^{v_q(n!)},
\]
where $v_q(n!)$ is the exponent of $q$ in the prime factorization of $n!$.
Proposition~\ref{prop:quartic-global} proves $4nT_1(n)^2<n!$ for $n\ge8$.
This comparison absorbs the factor $u^4\le n$ in $n!=u^4(v^4-1)$, extending the reduction beyond the classical equation studied in \cite{PollackShapiro}.
Its uniform character estimate and six short interval comparisons give the complete fourth-power and even-exponent results without an external computational certificate.

\item \emph{Fifth powers and multiples of five.}
An independent elementary proof of an initial range is given in Theorem~\ref{thm:quintic-elementary-range}, followed by the congruence argument for Corollary~\ref{cor:quintic-finite}.
Theorem~\ref{thm:quintic-analytic-detector} compares a quintic character kernel with ideals having small prime ideal factors, giving a detecting prime $q\le p^{19/20}$ for every sufficiently large prime $p$.
The small-prime-factor principle has a precedent in the cubic-residue discussion \cite{Cagarrinho}; the application here gives Theorem~\ref{thm:quintic-eventual}.
Under the specified GRH hypothesis, Lemma~\ref{lem:quintic-erh} gives a detecting prime $q<p$ for every relevant $p\ge1100$.
Together with the finite congruence argument, this proves the stated conditional result.
\end{enumerate}

Section~\ref{sec:valuation} gives the elementary reductions and the bounded-base theorem; Section~\ref{sec:cyclotomic} excludes power displacements at odd prime exponents.
Sections~\ref{sec:cubic} and \ref{sec:quartic} treat cubes and even exponents, and Section~\ref{sec:quintic} proves the fifth-power results.

\section{Elementary reductions}\label{sec:valuation}

We use standard valuation identities, Bertrand's postulate and quadratic reciprocity; see \cite{HardyWright}.
The factorial-specific reductions are proved below.

For a prime $p$ and a positive integer $u$, write $v_p(u)$ for the exponent of $p$ in the prime factorization of $u$.
We use the elementary identity
\begin{equation}\label{eq:valuation-difference}
v_p(u-v)=\min\{v_p(u),v_p(v)\}
\quad\text{if }v_p(u)\ne v_p(v),
\end{equation}
where the valuation of a nonzero negative integer means that of its absolute value.

\begin{lemma}\label{lem:valuation}
Let $n\ge4$ and $1\le d\le n$ be integers.
For every prime $p\le n/2$,
\begin{equation}\label{eq:valuation-bound}
v_p(n!)>v_p(d).
\end{equation}
\end{lemma}

\begin{proof}
Put $e=v_p(d)$.
If $e=0$, the result follows from $p\mid n!$.
If $e=1$, both $p$ and $2p$ occur among the factors of $n!$, so $v_p(n!)\ge2>e$.
Suppose that $e\ge2$.
Since $p^e\le d\le n$, counting multiples of $p$ and of $p^2$ gives
\begin{align*}
v_p(n!)
&\ge \left\lfloor\dfrac np\right\rfloor+
      \left\lfloor\dfrac n{p^2}\right\rfloor\\
&\ge p^{e-1}+p^{e-2}
 \ge 3\cdot2^{e-2}>e.
\end{align*}
The last inequality holds for $e=2$ and then follows inductively.
\end{proof}

\begin{lemma}\label{lem:dichotomy}
Let $n\ge4$, $m\ge1$, and $k\ge2$ be integers, and suppose that
\[
1\le d=|n!-m^k|\le n.
\]
Then either $d=a^k$ for some integer $a\ge1$, or $d$ is a prime $p$ satisfying $n/2<p\le n$.
In the first case, $m=ab$ for an integer $b\ge1$ with $\gcd(a,b)=1$.
\end{lemma}

\begin{proof}
For every prime divisor $p$ of $d$ with $p\le n/2$, Lemma~\ref{lem:valuation} and \eqref{eq:valuation-difference} give
\[
v_p(m^k)=v_p(d).
\]
Therefore every such exponent $v_p(d)$ is divisible by $k$.
If all prime divisors of $d$ are at most $n/2$, this proves $d=a^k$, including $a=1$ when $d=1$.
If $d$ has a prime divisor $p>n/2$, the inequality $d\le n<2p$ forces $d=p$.

Now suppose $d=a^k$.
Every prime divisor of $a$ is at most
$a\le n^{1/k}\le\sqrt n\le n/2$.
The valuation equality above gives $v_p(m)=v_p(a)$ for each such prime.
Hence $a\mid m$, and after writing $m=ab$, no prime divisor of $a$ divides $b$.
\end{proof}

The following is the standard perfect-power consequence of Bertrand's postulate; see \cite{HardyWright}.

\begin{lemma}\label{lem:not-power}
For every integer $n\ge2$, the integer $n!$ is not a nontrivial perfect power.
\end{lemma}

\subsection{Squares below a factorial}\label{sec:lower}

\begin{theorem}\label{thm:lower}
Let $n\ge4$ and $m\ge1$ be integers.
If $m^2<n!$, then
\begin{equation}\label{eq:lower}
n!-m^2>n.
\end{equation}
\end{theorem}

\begin{proof}
Assume, to the contrary, that $d=n!-m^2$ satisfies $1\le d\le n$.
By Lemma~\ref{lem:dichotomy}, $d$ is a square or a prime larger than $n/2$.

First suppose $d=a^2$.
The same lemma gives $m=ab$, so
\begin{equation}\label{eq:sum-squares}
n!=a^2(b^2+1).
\end{equation}
For $n\ge6$, Lemma~\ref{lem:valuation} applied to $p=3$ yields
$v_3(n!)>v_3(a^2)$.
For $n=4,5$, the restriction $a^2\le n$ gives $a\in\{1,2\}$, so the same strict inequality follows from $3\mid n!$ and $3\nmid a$.
Equation~\eqref{eq:sum-squares} therefore implies $3\mid b^2+1$.
This is impossible because a square is congruent to $0$ or $1$ modulo $3$.

Now suppose $d=p$ is a prime with $n/2<p\le n$.
Since $n\ge4$, we have $8\mid n!$, and $p$ is odd.
Thus $m$ is odd, and
\[
m^2=n!-p\equiv-p\pmod8
\]
forces $p\equiv7\pmod8$.
It follows that $(p-1)/2\equiv3\pmod4$.
An odd positive integer congruent to $3$ modulo $4$ has a prime divisor $q\equiv3\pmod4$.
Choose such a divisor of $(p-1)/2$.
Then $q<p\le n$ and $p\equiv1\pmod q$, whence
\[
m^2=n!-p\equiv-1\pmod q.
\]
The integer $-1$ is not a quadratic residue modulo a prime $q\equiv3\pmod4$, giving the final contradiction.
\end{proof}

Applying Theorem~\ref{thm:lower} to $(m^{k/2})^2$ also excludes every even power below $n!$ from distance at most $n$, for $n\ge4$.
The starting index is sharp, since $3!-2^2=2\le3$.

\subsection{Squares above a factorial}\label{sec:squares}

We next exclude the prime alternative in Lemma~\ref{lem:dichotomy} for squares above a factorial.
For an odd prime $p$, let $\leg{u}{p}$ denote the Legendre symbol.

We use the following standard consequence of multiplicativity, quadratic reciprocity, and the supplementary law for $2$; see \cite{HardyWright}.

\begin{lemma}\label{lem:nonresidue}
If $p$ is a prime with $p\equiv1\pmod8$, there is an odd prime $q<p$ such that $\leg{p}{q}=-1$.
\end{lemma}

\begin{proof}
Choose $1\le u<p$ with $\leg{u}{p}=-1$; such a $u$ exists because precisely half the nonzero residues modulo $p$ are squares.
By multiplicativity, some prime divisor $q$ of $u$ satisfies $\leg{q}{p}=-1$.
The supplementary law gives $\leg{2}{p}=1$, so $q$ is odd, and $q\le u<p$.
Since $p\equiv1\pmod4$, quadratic reciprocity yields $\leg{p}{q}=\leg{q}{p}=-1$.
\end{proof}

\begin{theorem}\label{thm:classification}
Let $n\ge4$ and $m\ge1$ be integers.
If $|n!-m^2|\le n$, then there are integers $a,b\ge1$ such that
\begin{equation}\label{eq:scaled}
m=ab,\qquad
m^2-n!=a^2,\qquad
a^2\le n,\qquad
n!=a^2(b^2-1),
\end{equation}
with $\gcd(a,b)=1$ and $b$ odd.
In addition,
\begin{equation}\label{eq:square-coprimality}
\gcd(b,n!)=1,\qquad \gcd(m,n!)=a,\qquad b\ge n+1.
\end{equation}
Conversely, any positive integers $a,b$ satisfying $a^2\le n$ and $n!=a^2(b^2-1)$ give a square $(ab)^2$ within distance $n$ above $n!$.
\end{theorem}

\begin{proof}
Equality $n!=m^2$ is excluded by Lemma~\ref{lem:not-power}.
Theorem~\ref{thm:lower} excludes $m^2<n!$, so $d=m^2-n!$ satisfies $1\le d\le n$.
By Lemma~\ref{lem:dichotomy}, either $d=a^2$ or $d=p$ is prime with $n/2<p\le n$.

Suppose $d=p$.
Reduction of $m^2=n!+p$ modulo $8$ gives $p\equiv1\pmod8$.
Choose $q<p$ as in Lemma~\ref{lem:nonresidue}.
Since $q\le n$, reduction modulo $q$ gives $m^2\equiv p\pmod q$, contradicting $\leg{p}{q}=-1$.
Therefore $d=a^2$.
Lemma~\ref{lem:dichotomy} now gives $m=ab$ and $\gcd(a,b)=1$, and rearrangement yields \eqref{eq:scaled}.
Finally, Lemma~\ref{lem:valuation} at $p=2$ shows that $n!/a^2=b^2-1$ is even, so $b$ is odd.
If a prime $q\le n$ divided $b$, the identity $n!=a^2(b^2-1)$ would force $q\mid a$, contradicting $\gcd(a,b)=1$.
Thus $b$ is coprime to $n!$, and $\gcd(m,n!)=a$ because $a\mid n!$.
Since $b>1$, it has a prime divisor greater than $n$, giving $b\ge n+1$.
The converse follows directly from $(ab)^2-n!=a^2$.
\end{proof}

For $a=1$, \eqref{eq:scaled} gives the equation in the Brocard--Ramanujan conjecture.
The values $a\ge2$ describe the other possible square displacements.

\subsection{Bases with bounded prime divisors}\label{sec:bounded-base}

For $m\ge2$, let $P^+(m)$ be its largest prime divisor, and put
\[
S_n=2^{v_2(n!)}.
\]
We will use the following elementary growth bound.

\begin{lemma}\label{lem:two-part-growth}
For every $n\ge8$,
\begin{equation}\label{eq:two-part-quadratic}
S_n>(n+3)(n-1).
\end{equation}
\end{lemma}

\begin{proof}
One has $S_8=S_9=128$ and $S_{n+2}\ge2S_n$.
For $f(n)=(n+3)(n-1)$,
\[
2f(n)-f(n+2)=n^2-2n-11>0\qquad(n\ge8).
\]
Induction with step two proves the assertion.
\end{proof}

The standard $2$-adic lifting-the-exponent formulas also give
\begin{equation}\label{eq:LTE-bound}
2^{v_2(B^E-\delta)}\le(B+1)E
\qquad(B\ge3\text{ odd},\ E\ge1,\ \delta\in\{-1,1\}).
\end{equation}
For odd $E$, the relevant $2$-part is that of $B-\delta$.
For even $E$, it is $2$ when $\delta=-1$, and at most $(B+1)2^{v_2(E)}$ when $\delta=1$.
These alternatives give \eqref{eq:LTE-bound} immediately.

\begin{theorem}\label{thm:bounded-base}
Let $n\ge8$ and $m,k\ge2$ be integers.
Then
\begin{equation}\label{eq:bounded-base-theorem}
P^+(m)\le n+2\quad\Longrightarrow\quad |n!-m^k|>n.
\end{equation}
\end{theorem}

\begin{proof}
Suppose that $P^+(m)\le n+2$ and $|n!-m^k|\le n$.
Lemma~\ref{lem:not-power} excludes equality, so write
\[
m^k=n!+\varepsilon d,\qquad 1\le d\le n,\qquad \varepsilon\in\{-1,1\}.
\]
Lemma~\ref{lem:dichotomy} gives two alternatives.

First suppose that $d=a^k$.
Then $m=ab$, $\gcd(a,b)=1$, and
\begin{equation}\label{eq:bounded-base-power}
n!=a^k(b^k-\varepsilon),\qquad a^k\le n.
\end{equation}
No prime at most $n$ divides $b$, since such a prime would also divide $a$, contradicting coprimality.
Also $b>1$, since $b=1$ would imply $n!\le2n$.
Thus every prime divisor of $b$ lies in $\{n+1,n+2\}$.
At most one of these consecutive integers is prime, so $b=q^t$ for a prime $q\in\{n+1,n+2\}$ and an integer $t\ge1$.

Suppose first that $q=n+1$.
Wilson's theorem and $q\mid m$ give
\[
\varepsilon a^k\equiv1\pmod q.
\]
Since $1\le a^k\le q-1$, either $\varepsilon=1$ and $a=1$, or $\varepsilon=-1$ and $a^k=n$.
In the latter case, an even $k$ would contradict Theorem~\ref{thm:lower}, applied to the square $(m^{k/2})^2$.
If $k$ is odd, then $k\ge3$ and $q=a^k+1$ is divisible by $a+1$, with $1<a+1<q$, again a contradiction.
If $q=n+2$, Wilson's theorem gives $n!\equiv1\pmod q$, hence
\[
\varepsilon a^k\equiv-1\pmod q.
\]
The bound $a^k\le q-2$ excludes $\varepsilon=1$, while $\varepsilon=-1$ forces $a=1$.

Consequently $a=1$ in every remaining case.
Put $E=kt$.
The inequalities
\[
q^E=n!+\varepsilon\le n!+1<(n+1)^n\le q^n
\]
give $E\le n-1$.
Taking $2$-parts and using \eqref{eq:LTE-bound}, we obtain
\begin{equation}\label{eq:bounded-base-two-part}
S_n=2^{v_2(q^E-\varepsilon)}\le(q+1)E\le(n+3)(n-1),
\end{equation}
contradicting \eqref{eq:two-part-quadratic}.

It remains to consider $d=p$, where $p$ is prime and $n/2<p\le n$.
Reduction of $m^k=n!+\varepsilon p$ shows that $p\mid m$ and that no other prime at most $n$ divides $m$.
We next exclude the two possible prime divisors just above $n$.
If $q=n+1$ is prime and divides $m$, Wilson's theorem gives
\[
n!\equiv-1\pmod q,\qquad \varepsilon p\equiv1\pmod q.
\]
For $1\le p\le n=q-1$, this would require $p=1$ or $p=n$.
The latter is impossible because $n=q-1\ge8$ is even.
If $q=n+2$ is prime and divides $m$, Wilson's theorem instead gives
\[
n!\equiv1\pmod q,\qquad \varepsilon p\equiv-1\pmod q,
\]
which would require $p=1$ or $p=n+1>n$.
Consequently $m=p^t$ for some $t\ge1$.
With $E=kt-1$, the original equation becomes
\[
\dfrac{n!}{p}=p^E-\varepsilon.
\]
By the strict arithmetic--geometric mean inequality and $n<2p$,
\[
n!<\left(\dfrac{n+1}{2}\right)^n\le p^n.
\]
Since $n!/p$ is an integer, it follows that
\[
p^E=\dfrac{n!}{p}+\varepsilon\le p^{n-1},\qquad E\le n-1.
\]
The prime $p$ is odd, so \eqref{eq:LTE-bound} yields
\[
S_n=2^{v_2(p^E-\varepsilon)}\le(p+1)E\le(n+1)(n-1),
\]
contradicting \eqref{eq:two-part-quadratic}.
\end{proof}

\begin{corollary}\label{cor:large-base-prime}
Every counterexample to Conjecture~\ref{conj:fsc} satisfies $P^+(m)>n+2$.
\end{corollary}

\section{Cyclotomic exclusion of power displacements}\label{sec:cyclotomic}

For integers $h\ge2$, define
\[
T(n,h)=\prod_{\substack{q\le n\text{ prime}\\q\equiv1\pmod h}}q^{v_q(n!)}.
\]
An empty product is $1$.
All logarithms in this section are natural logarithms.

\begin{lemma}\label{lem:EO}
For every odd prime $\ell$ and $n\ge7$,
\begin{equation}\label{eq:EO-odd}
T(n,2\ell)\le
\left(4\ell^{\ell/(\ell-1)}\right)^{\dfrac{n}{2\ell-1}(\log n-\log7+1)}.
\end{equation}
\end{lemma}

This is formula (Va) on page 248 of \cite{ErdosOblath}, with the same definition of $T$.
The estimate is applied to divisors of $n!/a^k$.

\begin{lemma}\label{lem:cyclotomic}
Let $\ell\ge3$ be an odd prime, $b\ge2$, and $\varepsilon\in\{-1,1\}$.
Put
\[
U=b^\ell-\varepsilon,\qquad
C=\dfrac{b^\ell-\varepsilon}{b-\varepsilon}.
\]
Then $C$ is an odd positive integer,
\begin{equation}\label{eq:C-lower}
C>\dfrac23 U^{(\ell-1)/\ell},
\end{equation}
and every prime divisor $q\ne\ell$ of $C$ satisfies $q\equiv1\pmod{2\ell}$.
Moreover, $v_\ell(C)\le1$.
Finally, $C>\ell$, except that $C=\ell=3$ when $(b,\ell,\varepsilon)=(2,3,-1)$.
\end{lemma}

The prime-support, exceptional-prime valuation, and exceptional equality assertions are the classical cyclotomic facts used in \cite[Section~2, pp.~248--250]{ErdosOblath}, specialized to the coprime pair $(b,1)$.

\begin{proof}[Proof of the lower bound \eqref{eq:C-lower}]
If $\varepsilon=1$, then $b-1<U^{1/\ell}$, so $C>U^{(\ell-1)/\ell}$.
If $\varepsilon=-1$, then $b+1\le\dfrac32b<\dfrac32U^{1/\ell}$, giving \eqref{eq:C-lower}.
\end{proof}

\begin{lemma}\label{lem:comparison}
For every $n\ge8$ and every odd prime $\ell$ with $2\ell+1\le n$,
\begin{equation}\label{eq:comparison-odd}
\ell T(n,2\ell)<\dfrac23\big((n-1)!\big)^{(\ell-1)/\ell}.
\end{equation}
\end{lemma}

\begin{proof}
First suppose $n\ge22$.
Put $L=\log n-\log7+1$ and
\[
f(t)=\log t+
\dfrac{\log4+\dfrac{t}{t-1}\log t}{2t-1}\,nL
\qquad\left(3\le t\le\dfrac{n-1}{2}\right).
\]
To see that $f$ is strictly decreasing, let $A(t)=\log4+\dfrac{t}{t-1}\log t$.
Then
\[
A'(t)\le\dfrac1{t-1},\qquad
\left(\dfrac{A(t)}{2t-1}\right)'
<-\dfrac{2\log t}{(2t-1)^2}.
\]
The latter inequality follows from
$\dfrac{2t-1}{t-1}\le\dfrac52<2\log4$ and $A(t)\ge\log4+\log t$.
Since $n\ge2t+1$, $L\ge1$, and $\log t>1$, we obtain
\[
f'(t)<\dfrac1t-\dfrac{2nL\log t}{(2t-1)^2}<0.
\]
Thus \eqref{eq:EO-odd} gives
\[
\log\big(\ell T(n,2\ell)\big)
\le\log3+c\,n(\log n-\log7+1),
\qquad c=\dfrac{\log4+\dfrac32\log3}{5}<0.607.
\]
Because $(\ell-1)/\ell\ge2/3$, \eqref{eq:comparison-odd} follows if
\[
G(n)=\dfrac23\log((n-1)!)-\log(9/2)
-c\,n(\log n-\log7+1)>0.
\]
The elementary estimates
\[
\begin{gathered}
\log(21!)>45.38,\qquad\log(9/2)<1.505,\\
3.09<\log22<3.092,\qquad\log7>1.945
\end{gathered}
\]
give
\[
G(22)>\dfrac23(45.38)-1.505
-0.607\cdot22(3.092-1.945+1)>0.077.
\]
For $n\ge22$, using $\log(1+1/n)<1/n$,
\begin{align*}
G(n+1)-G(n)
&>\left(\dfrac23-c\right)\log n-\dfrac cn-c(2-\log7)\\
&>0.059\cdot3.09-\dfrac{0.607}{22}-0.607\cdot0.055>0.121.
\end{align*}
Hence $G(n)>0$ for all $n\ge22$.

For $8\le n\le21$, only $\ell=3,5,7$ are admissible.
The values in Table~\ref{tab:cyclotomic} follow directly from Legendre's formula.
For $\ell=3$ and $8\le n\le12$, one has $T(n,6)=7$ and $7!>2^{12}$, so
\[
3T(n,6)=21<\dfrac23\,2^8
<\dfrac23\big((n-1)!\big)^{2/3}.
\]
For $13\le n\le21$, use $T(n,6)\le84\,721$, $9\cdot84\,721<3\cdot2^{18}$, and $12!>2^{28}$.
These give
\[
\bigl(9T(n,6)\bigr)^3
<27\cdot2^{54}<2^{59}<8(12!)^2
\le8\big((n-1)!\big)^2,
\]
which is the cube of the required comparison.
For $\ell=5$ or $7$, admissibility gives $n\ge11$ and the table gives $\ell T(n,2\ell)\le55$.
Since $10!>2^{20}$ and $(\ell-1)/\ell\ge4/5$, we have
\[
\dfrac23\big((n-1)!\big)^{(\ell-1)/\ell}
>\dfrac23\,2^{16}>55.
\]
This proves \eqref{eq:comparison-odd} throughout the finite range.

All decimal inequalities above are strict rational bounds.
Taking $J=24$, after scaling each argument by a power of $2$ into $[1,2]$, certifies all the numerical inequalities displayed above through
\[
\log x=2\sum_{j=0}^{J-1}\dfrac{y^{2j+1}}{2j+1}+R_J,\qquad
y=\dfrac{x-1}{x+1},\qquad
0\le R_J\le\dfrac{2y^{2J+1}}{(2J+1)(1-y^2)}.
\]
Direct substitution in this finite rational enclosure gives each displayed endpoint; no external numerical certificate is used.
\end{proof}

\begin{table}[ht]
\centering
\caption{Progression parts in the finite range of Lemma~\ref{lem:comparison}. A dash indicates an inadmissible exponent.}
\label{tab:cyclotomic}
\begin{tabular}{rrrr}
\toprule
$n$ & $T(n,6)$ & $T(n,10)$ & $T(n,14)$\\
\midrule
8--10 & 7 & --- & ---\\
11--12 & 7 & 11 & ---\\
13 & 91 & 11 & ---\\
14 & 637 & 11 & ---\\
15--18 & 637 & 11 & 1\\
19--20 & 12\,103 & 11 & 1\\
21 & 84\,721 & 11 & 1\\
\bottomrule
\end{tabular}
\end{table}

\FloatBarrier
\begin{theorem}\label{thm:odd-shift}
For $n\ge8$ and an odd prime $\ell$, there are no positive integers $a,b$ satisfying
\[
n!=a^\ell(b^\ell-\varepsilon),\qquad a^\ell\le n,\qquad
\varepsilon\in\{-1,1\}.
\]
Consequently, if $m\ge2$ and $|n!-m^\ell|\le n$, then
\[
|n!-m^\ell|=p\qquad\text{for a prime }n/2<p\le n.
\]
\end{theorem}

\begin{proof}
Suppose such $a,b$ exist, and put $U=n!/a^\ell=b^\ell-\varepsilon$.
Since $U\ge(n-1)!$, we have $b\ge2$.
For $C$ as in Lemma~\ref{lem:cyclotomic}, its prime support and $C\mid n!$ imply
\[
C\le\ell T(n,2\ell).
\]
If $n<2\ell+1$, then $T(n,2\ell)=1$, whereas Lemma~\ref{lem:cyclotomic} gives $C>\ell$, with the sole exception $b=2$, $\ell=3$, $\varepsilon=-1$.
The exceptional value $\ell=3$ would require $n<7$, contrary to $n\ge8$.
Thus $n\ge2\ell+1$.

Now Lemmas~\ref{lem:cyclotomic} and \ref{lem:comparison} give
\[
C>\dfrac23U^{(\ell-1)/\ell}
\ge\dfrac23\big((n-1)!\big)^{(\ell-1)/\ell}
>\ell T(n,2\ell),
\]
a contradiction.
The final assertion follows from Lemmas~\ref{lem:not-power} and \ref{lem:dichotomy}.
\end{proof}

\section{Cubic separation}\label{sec:cubic}

The remaining prime displacement for exponent three can be excluded by cubic reciprocity and the geometry of the Eisenstein lattice.
Both signs of the displacement are treated at once.

Put $E=\mathbb Z[\omega]$, where $\omega^2+\omega+1=0$, and put $\lambda=1-\omega$.
Under the complex embedding,
\[
N(a+b\omega)=|a+b\omega|^2=a^2-ab+b^2,
\qquad \operatorname{covol}(E)=\dfrac{\sqrt3}{2}.
\]
The ring $E$ is a unique factorization domain, with units $\{\pm1,\pm\omega,\pm\omega^2\}$.
The prime above $3$ is $\lambda$; rational primes congruent to $2$ modulo $3$ remain prime, and those congruent to $1$ modulo $3$ split into conjugate primes.

For an Eisenstein prime $\pi$ not dividing $3$, the cubic residue symbol $(\alpha/\pi)_3$ is the unique element of $\mu_3=\{1,\omega,\omega^2\}$ satisfying
\[
\left(\dfrac{\alpha}{\pi}\right)_3
\equiv\alpha^{(N(\pi)-1)/3}\pmod\pi,
\qquad \pi\nmid\alpha.
\]
Extend the symbol multiplicatively in the denominator through its prime ideals.
An element coprime to $3$ is called primary when it is $1$ modulo $3E$; it has a unique primary associate.
Cubic reciprocity \cite[Chapter~9, Section~3, Theorem~1]{IrelandRosen} gives
\begin{equation}\label{eq:cubic-reciprocity}
\left(\dfrac{\pi}{\theta}\right)_3
=\left(\dfrac{\theta}{\pi}\right)_3
\end{equation}
for primary primes $\pi,\theta$ of distinct norms.
The normalization of a primary associate as $1$ modulo $3E$, rather than $-1$, does not change the symbols because $-1$ is a cube.

The next statement is the sharp planar Hermite bound $\gamma_2=2/\sqrt3$; see \cite[equation~(17) and Theorem~3.4(i)]{Neumaier}.

\begin{lemma}\label{lem:planar-lattice}
Every planar Euclidean lattice of covolume $D$ contains a nonzero vector of squared length at most $2D/\sqrt3$.
\end{lemma}

\begin{lemma}\label{lem:cubic-nonresidue}
Let $q>3$ be a rational prime satisfying $q\equiv1$ or $8\pmod9$.
There is a rational prime $r<q$, with $r\equiv1\pmod3$, such that $q$ is not a cube modulo $r$.
\end{lemma}

\begin{proof}
For a unit residue modulo $qE$, define the cubic Jacobi character
\begin{equation}\label{eq:cubic-character}
\chi_q(\alpha)=\left(\dfrac{\alpha}{q}\right)_3
=\prod_{\mathfrak p\mid qE}
\left(\dfrac{\alpha}{\mathfrak p}\right)_3.
\end{equation}
Here the product has two factors if $q\equiv1\pmod3$, and one factor if $q\equiv2\pmod3$.
The character is nontrivial on $(E/qE)^\times$.
Indeed, in the inert case this group is $\mathbb F_{q^2}^\times$, and the character is the power map with exponent $(q^2-1)/3$.
In the split case the Chinese remainder theorem gives $E/qE\cong\mathbb F_q\times\mathbb F_q$, and either of the two nontrivial local characters can be varied independently.

For every rational integer $t$ not divisible by $q$,
\begin{equation}\label{eq:cubic-trivial-factors}
\chi_q(t)=1,\qquad \chi_q(\omega)=1,\qquad \chi_q(\lambda)=1.
\end{equation}
To prove the first equality, conjugation fixes $t$ and permutes the prime ideals over $q$, while conjugating the symbol values.
Thus $\chi_q(t)$ is a real cube root of unity, hence is $1$.
For the second equality, the local exponent is $(q-1)/3$ if $q\equiv1\pmod9$, and $(q^2-1)/3$ if $q\equiv8\pmod9$.
In either case it is divisible by $3$.
Finally, $\lambda^2=-3\omega$ gives $\chi_q(\lambda)^2=1$, so $\chi_q(\lambda)=1$.
In particular, $\chi_q$ is trivial on $E^\times$.
The distinction between $E^\times$ and $(E/qE)^\times$ is essential here.

Suppose, contrary to the conclusion, that $q$ is a cube modulo every rational prime $r<q$ with $r\equiv1\pmod3$.
Let $\pi$ be a primary Eisenstein prime of norm such an $r$.
Set $q^*=q$ if $q\equiv1\pmod3$, and $q^*=-q$ otherwise.
Then $q^*$ is primary and its prime factorization consists of primary primes, with no remaining unit factor: their product is primary, and the only primary Eisenstein unit is $1$.
Applying \eqref{eq:cubic-reciprocity} to these factors gives
\[
\chi_q(\pi)
=\left(\dfrac{\pi}{q^*}\right)_3
=\left(\dfrac{q^*}{\pi}\right)_3
=\left(\dfrac{q}{\pi}\right)_3
=1.
\]
The third equality uses that $-1$ is a cube, and the last uses $E/(\pi)\cong\mathbb F_r$.
All other Eisenstein prime factors of norm less than $q$ are rational primes or associates of $\lambda$, and have character value $1$ by \eqref{eq:cubic-trivial-factors}.
Unique factorization therefore implies
\begin{equation}\label{eq:small-cubic-trivial}
\chi_q(\alpha)=1\qquad\text{whenever }0<N(\alpha)<q.
\end{equation}
Such elements are units modulo $qE$, since every prime ideal over $q$ has norm at least $q$.

Choose $\gamma\in(E/qE)^\times$ with $\chi_q(\gamma)\ne1$.
Every nonzero rational scalar multiple of $\gamma$ has the same character value, by \eqref{eq:cubic-trivial-factors}.
The inverse image $L$ of the one-dimensional $\mathbb F_q$-subspace $\mathbb F_q\gamma$ is a lattice of index $q$ in $E$.
Its covolume is $q\sqrt3/2$.
Lemma~\ref{lem:planar-lattice} supplies $0\ne\alpha\in L$ with $N(\alpha)\le q$.
The reduction of $\alpha$ is nonzero, because every nonzero element of $qE$ has norm at least $q^2$.
Hence $\alpha\bmod qE=t\gamma$ for some $t\in\mathbb F_q^\times$.
This residue is a unit; its conjugate is also a unit, so the rational norm $N(\alpha)=\alpha\overline\alpha$ is nonzero modulo $q$.
Thus $q\nmid N(\alpha)$, excluding $N(\alpha)=q$.
We obtain
\[
0<N(\alpha)<q,\qquad
\chi_q(\alpha)=\chi_q(t)\chi_q(\gamma)=\chi_q(\gamma)\ne1,
\]
contradicting \eqref{eq:small-cubic-trivial}.
\end{proof}

\begin{samepage}
\begin{theorem}\label{thm:cubic-separation}
For every integer $n\ge6$ and every integer $M\ge2$,
\begin{equation}\label{eq:cubic-separation}
|n!-M^3|>n.
\end{equation}
Consequently, for $n\ge6$, $m\ge2$, and every positive integer $k$ divisible by $3$, we have $|n!-m^k|>n$.
\end{theorem}
\end{samepage}

\begin{proof}
First suppose $n\ge8$ and $|n!-M^3|\le n$.
The displacement is nonzero by Lemma~\ref{lem:not-power}.
Lemma~\ref{lem:dichotomy} and Theorem~\ref{thm:odd-shift}, with $\ell=3$, exclude a cubic displacement for both signs.
Thus
\[
M^3=n!+\varepsilon q,\qquad
\varepsilon\in\{-1,1\},\qquad \dfrac n2<q\le n,
\]
where $q\ge5$ is prime.
Since $9\mid n!$ and a unit cube modulo $9$ is $1$ or $-1$, we have $q\equiv1$ or $8\pmod9$.
For every rational prime $r<q$, reduction modulo $r$ gives $q\equiv(\varepsilon M)^3\pmod r$.
This contradicts Lemma~\ref{lem:cubic-nonresidue}.

At $n=6$, the neighboring cubes are $8^3<6!<9^3$, at distances $208$ and $9$.
At $n=7$, they are $17^3<7!<18^3$, at distances $127$ and $792$.
Monotonicity of cubes proves \eqref{eq:cubic-separation} at both indices.
Finally, write $m^k=(m^{k/3})^3$ when $3\mid k$.
\end{proof}

The starting index in Theorem~\ref{thm:cubic-separation} is sharp, since $5!+5=5^3$.

\section{Fourth powers and even exponents}\label{sec:quartic}

The product obstruction below reduces separation from fourth powers to an inequality involving only the factorial index.
A character estimate proves that inequality uniformly for $n\ge121$, and six elementary interval bounds cover $8\le n\le120$.
The argument requires no separate computational certificate.
Put $T_1(n)=T(n,4)$.
All results in this section are unconditional.

\begin{proposition}\label{prop:quartic-product}
Let $n\ge8$ and $m\ge2$ be integers.
If $|n!-m^4|\le n$, then
\begin{equation}\label{eq:quartic-necessary}
n!<4nT_1(n)^2.
\end{equation}
Consequently, the exact inequality $4nT_1(n)^2\le n!$ proves separation from every fourth power at that index.
\end{proposition}

\begin{proof}
Theorem~\ref{thm:lower} excludes a lower fourth power because $m^4=(m^2)^2$.
Apply Theorem~\ref{thm:classification} with square base $m^2$ to obtain
\[
m^2=ab,\qquad n!=a^2(b^2-1),\qquad a^2\le n,\qquad \gcd(a,b)=1,
\]
where $b$ is odd.
Since $ab$ is a square and its factors are coprime, write $a=u^2$ and $b=v^2$.
Then
\[
m=uv,\qquad n!=u^4(v^4-1),\qquad u^4\le n,
\]
and $v$ is odd.
Put $C=v^2+1$.
Every odd prime divisor $q$ of $C$ has $q\equiv1\pmod4$, since $-1$ is a square modulo $q$.
Also $v_2(C)=1$ and $C\mid n!$, so $C\le2T_1(n)$.
Since $C^2>v^4-1=n!/u^4$, we obtain
\[
n!<u^4C^2\le4u^4T_1(n)^2\le4nT_1(n)^2.\qedhere
\]
\end{proof}

\subsection{A character estimate modulo four}

All logarithms are natural, and the finite decimals below denote exact rational endpoints.
For completeness, the numerical bounds in this section can be obtained from four terms of the logarithm series.
For $1\le u\le2$, put $z=(u-1)/(u+1)$ and
\[
s(u)=2\left(z+\dfrac{z^3}{3}+\dfrac{z^5}{5}+\dfrac{z^7}{7}\right),
\qquad r(u)=\dfrac{2z^9}{9(1-z^2)}.
\]
The geometric remainder gives
\begin{equation}\label{eq:quartic-log-series}
0\le\log u-s(u)\le r(u).
\end{equation}
If a rational $t\ge1$ is written as $t=2^a u$, with $a\ge0$ an integer and $1\le u\le2$, then
\begin{equation}\label{eq:quartic-log-enclosure}
\begin{aligned}
a s(2)+s(u)&\le\log t,\\
\log t&\le a\bigl(s(2)+r(2)\bigr)+s(u)+r(u).
\end{aligned}
\end{equation}
In particular, these rational bounds give
\begin{equation}\label{eq:quartic-basic-logs}
\begin{gathered}
0.693<\log2<0.694,\qquad 1.098<\log3<1.099,\\
\log7>1.945,\qquad 2.397<\log11<2.4.
\end{gathered}
\end{equation}

Let $\mu$ be the M\"obius function and $\Lambda$ the von Mangoldt function.
Write $\chi=\chi_4$ for the real character modulo four, so that $\chi(r)=0$ for even $r$ and $\chi(r)=(-1)^{(r-1)/2}$ for odd $r$.
Define
\[
\begin{gathered}
C(x)=\sum_{r\le x}\dfrac{\chi(r)\Lambda(r)}r,\\
L=\sum_{r\ge1}\dfrac{\chi(r)}r=\dfrac\pi4,
\qquad M=\sum_{r\ge1}\dfrac{\chi(r)\log r}r.
\end{gathered}
\]
Both infinite series converge by alternation.

\begin{lemma}\label{lem:quartic-character}
For every real $x\ge6$,
\begin{equation}\label{eq:quartic-character-bound}
C(x)<\dfrac{93}{500}+\dfrac{\log x+83/360}{x}.
\end{equation}
\end{lemma}

\begin{proof}
Put
\[
Q(t)=\sum_{\substack{d\le t\\d\text{ odd}}}|\mu(d)|.
\]
An odd squarefree integer is not divisible by $9$.
The admissible residues modulo $18$ are $1,3,5,7,11,13,15,17$; counting them in complete blocks and the final partial block gives
\begin{equation}\label{eq:quartic-squarefree}
Q(t)\le\dfrac49t+1\qquad(t\ge1).
\end{equation}

The function $f(t)=\log t/t$ is decreasing and convex for $t\ge7$, since $f''(t)=(2\log t-3)/t^3>0$ there.
The alternating tail $f(7)-f(9)+f(11)-\cdots$ is therefore at least $f(7)/2$.
Indeed, for a decreasing convex sequence $u_j\to0$,
\[
2\sum_{j\ge0}(-1)^j u_j-u_0
=\sum_{j\ge0}(u_{2j}-2u_{2j+1}+u_{2j+2})\ge0.
\]
It follows from \eqref{eq:quartic-basic-logs} that
\[
M\le-\dfrac{\log3}{3}+\dfrac{\log5}{5}-\dfrac{\log7}{14}
<-\dfrac9{50};
\]
here $\log5<1.61$ follows from \eqref{eq:quartic-log-enclosure}.
As $0<L=\pi/4<4/5$, we obtain $M/L<-9/40$.

Put $y=x/3$ and
\[
Z(y)=\sum_{d\le y}\dfrac{\mu(d)\chi(d)}d.
\]
The identity $\mu*1=\delta_1$ gives the finite convolution
\[
\sum_{d\le y}\dfrac{\mu(d)\chi(d)}d
\sum_{r\le y/d}\dfrac{\chi(r)}r=1.
\]
The alternating-series remainder satisfies
\[
\left|L-\sum_{r\le t}\dfrac{\chi(r)}r\right|\le\dfrac1t
\qquad(t\ge1).
\]
Consequently, \eqref{eq:quartic-squarefree} yields
\begin{equation}\label{eq:quartic-mobius-lower}
LZ(y)\ge1-\dfrac{Q(y)}y\ge\dfrac59-\dfrac1y>0,
\end{equation}
where the last inequality uses $y\ge2$.

The identity $\Lambda=\mu*\log$ gives
\[
C(x)=\sum_{d\le x/3}\dfrac{\mu(d)\chi(d)}d
\sum_{r\le x/d}\dfrac{\chi(r)\log r}r.
\]
The outer sum stops at $x/3$ because the inner sum vanishes when $x/d<3$.
For $t\ge3$, the alternating-series estimate and the monotonicity of $\log t/t$ give
\[
\left|\sum_{r\le t}\dfrac{\chi(r)\log r}r-M\right|
\le\dfrac{\log t}{t}.
\]
Hence
\begin{equation}\label{eq:quartic-convolution-bound}
C(x)\le MZ(y)+\dfrac1x
\sum_{d\le y}|\mu(d)\chi(d)|\log(x/d).
\end{equation}
Partial summation and \eqref{eq:quartic-squarefree} imply
\begin{align*}
\sum_{d\le y}|\mu(d)\chi(d)|\log(x/d)
&=Q(y)\log3+\int_1^y\dfrac{Q(t)}t\,dt\\
&\le\dfrac{4x(1+\log3)}{27}+\log x-\dfrac49.
\end{align*}
Combining \eqref{eq:quartic-mobius-lower}, $M/L<-9/40$ and \eqref{eq:quartic-convolution-bound}, we obtain
\begin{align*}
C(x)&<\dfrac{4(1+\log3)}{27}-\dfrac18
+\dfrac{\log x+83/360}{x}\\
&<\dfrac{93}{500}+\dfrac{\log x+83/360}{x},
\end{align*}
where the last step uses $\log3<1.099$.
\end{proof}

\subsection{The factorial product comparison}

Define
\[
\begin{gathered}
A(n)=\sum_{r\le n}\dfrac{\Lambda(r)}r,
\qquad A_1(n)=\sum_{\substack{r\le n\\r\equiv1\pmod4}}\dfrac{\Lambda(r)}r,\\
S_1(n)=\sum_{\substack{q\equiv1\pmod4\text{ prime},\ j\ge1\\q^j\le n}}\dfrac{\log q}{q^j},\\
E_3(n)=\sum_{\substack{q\equiv3\pmod4\text{ prime},\ j\ge1\\q^{2j}\le n}}\dfrac{\log q}{q^{2j}}.
\end{gathered}
\]
Thus $A_1(n)=S_1(n)+E_3(n)$.
Also put
\[
\psi(x)=\sum_{r\le x}\Lambda(r),
\qquad
\mathcal F(n)=\sum_{r\le n}\left\{\dfrac nr\right\}\Lambda(r).
\]
We use the unconditional global estimate of Rosser and Schoenfeld \cite[Theorem~12, equation~(3.35)]{RosserSchoenfeld},
\begin{equation}\label{eq:RS-psi}
\psi(x)<1.03883x\qquad(x>0).
\end{equation}

\begin{lemma}\label{lem:quartic-fractional}
For every positive integer $n$,
\begin{equation}\label{eq:quartic-fractional}
\mathcal F(n)<\dfrac{39}{50}n.
\end{equation}
\end{lemma}

\begin{proof}
For $1\le r\le n$, telescoping gives the elementary majorant
\[
\left\{\dfrac nr\right\}
\le\sum_{\substack{j\ge1\\r\le(j+1)n/(j+2)}}\dfrac1{j(j+1)}.
\]
For $r\le2n/3$ the right side is $1$, and for $r=n$ both sides vanish.
Otherwise, choose $J\ge2$ with $Jn/(J+1)<r\le(J+1)n/(J+2)$.
The right side is $1/J$, whereas $(n-r)/r<1/J$.
Multiplying by $\Lambda(r)$ and summing, then using \eqref{eq:RS-psi}, gives
\begin{align*}
\mathcal F(n)
&\le\sum_{j=1}^{\infty}\dfrac{\psi((j+1)n/(j+2))}{j(j+1)}\\
&<1.03883n\sum_{j=1}^{\infty}\dfrac1{j(j+2)}\\
&=0.7791225n<\dfrac{39}{50}n.\qedhere
\end{align*}
\end{proof}

\begin{samepage}
\begin{proposition}\label{prop:quartic-global}
For every integer $n\ge8$,
\begin{equation}\label{eq:quartic-global}
4nT_1(n)^2<n!.
\end{equation}
For every $n\ge121$, one has the explicit bound
\begin{equation}\label{eq:quartic-explicit-decay}
\log\dfrac{n!}{T_1(n)^2}
>\dfrac{57}{500}n-\log n-\dfrac{81}{50}.
\end{equation}
\end{proposition}
\end{samepage}

\begin{proof}
First suppose $n\ge121$ and put $s=\lfloor\log_2 n\rfloor$.
Separating the powers of $2$ in $A(n)$ gives
\[
A(n)-2A_1(n)=\log2\,(1-2^{-s})-C(n).
\]
The four terms $3^2,3^4,7^2,11^2$ already yield
\begin{align*}
2E_3(n)&\ge2\log3\left(\dfrac19+\dfrac1{81}\right)
+\dfrac{2\log7}{49}+\dfrac{2\log11}{121}\\
&>\dfrac{387}{1000},
\end{align*}
by \eqref{eq:quartic-basic-logs}.
Legendre's formula gives $\log T_1(n)\le nS_1(n)$, while
\[
\log(n!)=nA(n)-\mathcal F(n).
\]
Consequently,
\begin{align*}
\log\dfrac{n!}{T_1(n)^2}
&\ge n\log2\,(1-2^{-s})-nC(n)+2nE_3(n)-\mathcal F(n)\\
&>\left(\log2-\dfrac{579}{1000}\right)n-\log n
-\left(2\log2+\dfrac{83}{360}\right)\\
&>\dfrac{57}{500}n-\log n-\dfrac{81}{50}.
\end{align*}
Here we used $n2^{-s}<2$, Lemmas~\ref{lem:quartic-character} and~\ref{lem:quartic-fractional}, and \eqref{eq:quartic-basic-logs}.
This proves \eqref{eq:quartic-explicit-decay}.
Subtracting $\log(4n)$, we obtain
\[
\log\dfrac{n!}{4nT_1(n)^2}
>\dfrac{57}{500}n-2\log n-\log4-\dfrac{81}{50}.
\]
At $n=121$, the bounds $\log121<4.8$ and $\log4<1.4$ give the strict lower bound
\[
\dfrac{57}{500}\cdot121-2\cdot4.8-1.4-\dfrac{81}{50}
=\dfrac{587}{500}>0.
\]
The derivative of the preceding scalar lower bound is $57/500-2/n>0$ for $n\ge121$.
Thus \eqref{eq:quartic-global} holds in the entire uniform range.

It remains to treat $8\le n\le120$.
For an upper endpoint $B$, define
\[
S(B)=\sum_{\substack{p\le B\text{ prime}\\p\equiv1\pmod4}}
\dfrac{\log p}{p-1}.
\]
The primes congruent to one modulo four up to $120$ are exactly those in Table~\ref{tab:quartic-prime-logs}.
Their completeness follows by trial division by $2,3,5,7$.
The displayed upper bounds $\log p<b_p$ follow from \eqref{eq:quartic-log-enclosure}.
Summing $b_p/(p-1)$ gives $S(B)<U$ for each row of Table~\ref{tab:quartic-intervals}.

\begin{table}[htbp]
\centering
\caption{Rational upper bounds $\log p<b_p$ for the relevant primes up to $120$.}
\label{tab:quartic-prime-logs}
\begin{tabular}{r*{7}{r}}
\toprule
$p$ & $5$ & $13$ & $17$ & $29$ & $37$ & $41$ & $53$\\
$b_p$ & $1.61$ & $2.57$ & $2.84$ & $3.37$ & $3.62$ & $3.72$ & $3.98$\\
\midrule
$p$ & $61$ & $73$ & $89$ & $97$ & $101$ & $109$ & $113$\\
$b_p$ & $4.12$ & $4.30$ & $4.49$ & $4.58$ & $4.62$ & $4.70$ & $4.73$\\
\bottomrule
\end{tabular}
\end{table}

For $n\le B$, Legendre's formula implies $\log T_1(n)\le nS(B)<nU$.
Stirling's standard lower bound $n!>\sqrt{2\pi n}(n/e)^n$ \cite{HardyWright} therefore gives
\begin{align*}
\log\dfrac{n!}{4nT_1(n)^2}
&>\left(n-\dfrac12\right)\log n-(1+2U)n
+\dfrac12\log(2\pi)-\log4\\
&>H_U(n),
\end{align*}
where
\[
H_U(n)=\left(n-\dfrac12\right)\log n-(1+2U)n-\dfrac12.
\]
The last step uses $\pi e>8$, which follows from $\pi>3$ and $e>8/3$.

\begin{table}[htbp]
\centering
\caption{Six interval bounds. In each row, $S(B)<U$, $\log A>\ell_A$ and $H_U(A)>h$.}
\label{tab:quartic-intervals}
\begin{tabular}{rrrrr}
\toprule
$A$ & $B$ & $U$ & $\ell_A$ & $h$\\
\midrule
$8$ & $12$ & $5/12$ & $2.079$ & $2/5$\\
$13$ & $16$ & $5/8$ & $2.564$ & $23/10$\\
$17$ & $28$ & $4/5$ & $2.833$ & $2$\\
$29$ & $52$ & $9/8$ & $3.367$ & $6/5$\\
$53$ & $100$ & $10/7$ & $3.970$ & $17/5$\\
$101$ & $120$ & $8/5$ & $4.615$ & $39$\\
\bottomrule
\end{tabular}
\end{table}

The bounds $\log A>\ell_A$ again follow from \eqref{eq:quartic-log-enclosure}.
Each row satisfies the rational inequalities
\[
\left(A-\dfrac12\right)\ell_A-(1+2U)A-\dfrac12\ge h>0,
\qquad \ell_A-2U-\dfrac1{2A}>0.
\]
Moreover,
\[
H_U'(x)=\log x-2U-\dfrac1{2x},\qquad
H_U''(x)=\dfrac1x+\dfrac1{2x^2}>0.
\]
Thus $H_U$ is increasing on each listed interval and is positive at its left endpoint.
The six intervals cover $8\le n\le120$ and complete the proof of \eqref{eq:quartic-global}.
\end{proof}

\begin{theorem}\label{thm:quartic-complete}
For every integer $n\ge8$ and every integer $m\ge2$,
\[
|n!-m^4|>n.
\]
\end{theorem}

\begin{proof}
A failure would imply $n!<4nT_1(n)^2$ by Proposition~\ref{prop:quartic-product}, contradicting Proposition~\ref{prop:quartic-global}.
\end{proof}

\begin{corollary}\label{cor:large-even}
For every $n\ge8$, every even $k\ge4$, and every $m\ge2$, we have $|n!-m^k|>n$.
\end{corollary}

\begin{proof}
Suppose the inequality failed.
Because $k$ is even, Theorem~\ref{thm:classification} excludes a prime displacement.
If $k$ has an odd prime divisor $\ell$, write $m^k=(m^{k/\ell})^\ell$.
Lemma~\ref{lem:dichotomy} makes the displacement an $\ell$-th power, contrary to Theorem~\ref{thm:odd-shift}.
Otherwise $k$ is a power of $2$ and, since $k\ge4$, is divisible by $4$.
Apply Theorem~\ref{thm:quartic-complete} to $(m^{k/4})^4$.
\end{proof}

\section{Fifth-power separation}\label{sec:quintic}

We first give an elementary proof on $8\le n\le100$ and then use a short congruence table to extend the unconditional finite interval to $1100$.
An unconditional analytic argument treats all sufficiently large indices, while a GRH argument based on the exact conductor covers the full range $n\ge8$.
The displacement prime is denoted by $p$, while $q$ denotes a prime detecting a failed fifth-power residue condition.
For a rational prime $p\ne5$, put
\[
K=\mathbb Q(\zeta_5),\qquad L_p=K(p^{1/5}),
\]
where $\zeta_5$ is a primitive fifth root of unity.

\subsection{An elementary proof for the initial interval}\label{sec:quintic-elementary}

The next proof uses only the elementary valuation dichotomy, the algebraic prime-divisor description in Lemma~\ref{lem:cyclotomic}, and explicit integer inequalities and congruences.
In particular, this initial proof is independent of the Erd\H{o}s--Obl\'ath estimate and of the later analytic arguments.

\begin{theorem}\label{thm:quintic-elementary-range}
For every integer $8\le n\le100$, every $m\ge2$ and every positive integer $k$ divisible by $5$,
\[
|n!-m^k|>n.
\]
\end{theorem}

\begin{proof}
It suffices to treat $k=5$.
For $n=8,9,10$, the neighboring fifth powers have bases $(8,9)$, $(12,13)$ and $(20,21)$, respectively.
Their minimum distances from $n!$ are $7552$, $8413$ and $428800$, respectively, all exceeding $n$.
Hence suppose that $11\le n\le100$ and that
\[
m^5=n!+\varepsilon d,\qquad 1\le d\le n,\qquad \varepsilon\in\{-1,1\}.
\]
Zero displacement is excluded by Lemma~\ref{lem:not-power}.
By Lemma~\ref{lem:dichotomy}, either $d=a^5$ or $d=p$ is prime with $n/2<p\le n$.

Suppose first that $d=a^5$.
Then $m=ab$, $a^5\le n$, and $b\ge2$, since $b=1$ would imply $n!\le2n$.
Consequently
\[
D=b^5-\varepsilon=\dfrac{n!}{a^5}\ge(n-1)!.
\]
Consider the cyclotomic factor
\[
C=\dfrac{b^5-\varepsilon}{b-\varepsilon}
=b^4+\varepsilon b^3+b^2+\varepsilon b+1.
\]
This is an odd divisor of $n!$.
By the prime-divisor assertions of Lemma~\ref{lem:cyclotomic}, every prime divisor of $C$ other than $5$ is congruent to $1$ modulo $10$, and $v_5(C)\le1$.
Also
\begin{equation}\label{eq:quintic-elementary-C-lower}
C\ge b^4-b^3+b^2-b+1>b^3(b-1).
\end{equation}

For $11\le n\le30$, the only prime at most $n$ congruent to $1$ modulo $10$ is $11$.
Since $v_{11}(30!)=2$, we obtain
\[
C\le5\cdot11^2=605.
\]
On the other hand, $D\ge10!>5^5+1$, so $b\ge6$.
The lower bound
\[
C\ge b^3(b-1)+b(b-1)+1\ge1111
\]
is a contradiction.

For $31\le n\le100$, the possible primes congruent to $1$ modulo $10$ are $11,31,41,61,71$.
Their exponents in $100!$ are $9,3,2,1,1$, respectively, so
\begin{equation}\label{eq:quintic-elementary-C-upper}
C\le5\cdot11^9\cdot31^3\cdot41^2\cdot61\cdot71
<2^{3+36+15+12+6+7}=2^{79}.
\end{equation}
Here each factor is bounded by the indicated power of $2$.
The integer inequality $15!=1{,}307{,}674{,}368{,}000>2^{40}$ gives
\[
30!>2^{40}\cdot16^{15}=2^{100}.
\]
Thus $D\ge30!>2^{100}$, forcing $b\ge2^{20}$.
Equation~\eqref{eq:quintic-elementary-C-lower} now implies
\[
C>2^{60}(2^{20}-1)>2^{79},
\]
contradicting \eqref{eq:quintic-elementary-C-upper}.
This excludes every fifth-power displacement in the initial interval by elementary means.

It remains to consider $d=p$ prime, where $7\le p\le100$.
Since $p\ne5$, reduction of $m^5=n!+\varepsilon p$ modulo $5$ gives $5\nmid m$.
As $25\mid n!$, Euler's theorem gives
\[
p^4\equiv m^{20}\equiv1\pmod{25}.
\]
The roots are $1,7,18,24$ modulo $25$.
Among the primes $7\le p\le100$, only $p=7$ and $p=43$ lie in those classes.
If $p=7$, reduction modulo $11$ and squaring give
\[
m^{10}\equiv7^2\equiv5\pmod{11},
\]
contrary to Fermat's theorem, since $11\nmid m$.
If $p=43$, then $n\ge43$, so reduction modulo $31$ and taking the sixth power give
\[
m^{30}\equiv43^6\equiv12^6\equiv2\pmod{31},
\]
again contrary to Fermat's theorem, since $31\nmid m$.
Both signs $\varepsilon$ disappear in these even powers.
This proves the assertion for $k=5$.
For $5\mid k$, replace $m$ by $m^{k/5}$.
\end{proof}

\subsection{Residue conditions and Kummer extensions}

\begin{lemma}\label{lem:quintic-residue-conditions}
Let $n\ge11$ and $M\ge2$ be integers with $|n!-M^5|\le n$.
Then
\begin{equation}\label{eq:quintic-shift}
M^5=n!+\varepsilon p,\qquad \varepsilon\in\{-1,1\},\qquad n/2<p\le n,
\end{equation}
where $p\ne5$ is prime.
Moreover,
\begin{equation}\label{eq:quintic-mod25}
p^4\equiv1\pmod{25},
\end{equation}
and, for every prime $q\le n$ with $q\ne p$ and $q\equiv1\pmod5$,
\begin{equation}\label{eq:quintic-residue}
p^{(q-1)/5}\equiv1\pmod q.
\end{equation}
\end{lemma}

\begin{proof}
Lemma~\ref{lem:not-power} excludes zero displacement.
Lemma~\ref{lem:dichotomy} and Theorem~\ref{thm:odd-shift} exclude a fifth-power displacement, giving \eqref{eq:quintic-shift}.
Since $n\ge11$, one has $p\ge7$.
For each prime $q\le n$ with $q\ne p$, reduction of \eqref{eq:quintic-shift} modulo $q$ gives $q\nmid M$.
In particular, $5\nmid M$.
Since $25\mid n!$, the congruence $M^5\equiv\varepsilon p\pmod{25}$, raised to the fourth power, gives $p^4\equiv M^{20}\equiv1\pmod{25}$ by Euler's theorem.
For $q\equiv1\pmod5$, raise $M^5\equiv\varepsilon p\pmod q$ to the even power $(q-1)/5$ and apply Fermat's theorem.
This proves \eqref{eq:quintic-residue}.
\end{proof}

\begin{lemma}\label{lem:quintic-fields}
Let $p\ne5$ be a rational prime, $K=\mathbb Q(\zeta_5)$ and $L_p=K(p^{1/5})$.
Then $L_p/\mathbb Q$ is Galois, $[L_p:\mathbb Q]=20$, and
\[
H=\operatorname{Gal}(L_p/K)\cong C_5
\]
is a normal subgroup of $\operatorname{Gal}(L_p/\mathbb Q)$.
The four nontrivial Hecke characters of $L_p/K$ have a common conductor $\mathfrak f$, supported above $5p$, with
\begin{equation}\label{eq:quintic-analytic-conductor}
\Nm\mathfrak f\le5^5p^4.
\end{equation}
\end{lemma}

\begin{proof}
The field $K$ has degree four and absolute discriminant $5^3$.
Every prime ideal of $K$ above $p$ has valuation one on $p$, because $p\ne5$ is unramified in $K$.
Thus $X^5-p$ is Eisenstein at any such prime ideal, giving $[L_p:K]=5$.
Since $K$ contains the fifth roots of unity, $L_p$ is the splitting field of $X^5-p$.
The assertions about the degree and Galois groups follow.

The relative field discriminant divides the discriminant of the order generated by $p^{1/5}$.
The polynomial discriminant is $5^5p^4$, so
\[
\Nm\mathfrak d_{L_p/K}\le5^{20}p^{16}.
\]
The nontrivial characters have the same kernel on every local unit filtration and hence have the same conductor $\mathfrak f$.
The conductor--discriminant formula \cite[Chapter~V, Theorem~3.27]{Milne} gives $\mathfrak d_{L_p/K}=\mathfrak f^4$, proving \eqref{eq:quintic-analytic-conductor} and the asserted support.
\end{proof}

\begin{lemma}\label{lem:quintic-exact-conductor}
Let $p\ne5$ be a rational prime satisfying $p^4\equiv1\pmod{25}$.
Then $L_p/K$ splits locally at the prime above $5$.
Every nontrivial Hecke character of $L_p/K$ has finite conductor
\begin{equation}\label{eq:quintic-exact-conductor}
\mathfrak f=p\mathcal O_K,\qquad \Nm\mathfrak f=p^4.
\end{equation}
Moreover,
\begin{equation}\label{eq:quintic-exact-discriminant}
|D_{L_p}|=5^{15}p^{16}.
\end{equation}
\end{lemma}

\begin{proof}
The four possible residues $p\pmod{25}$ are the fifth powers of $a=1,2,3,4$ modulo $25$.
Choose the corresponding $a$ and put
\[
g(T)=\dfrac{(a+5T)^5-p}{25}\in\mathbb Z_5[T].
\]
Modulo $5$, this polynomial is linear with nonzero coefficient $a^4$.
Hensel's lemma gives $t\in\mathbb Z_5$ with $g(t)=0$.
Thus $p$ is a fifth power in $\mathbb Q_5$.
The completion of $K$ above $5$ contains $\zeta_5$, so $X^5-p$ splits there and $L_p/K$ is locally trivial at that prime.

At every prime ideal $\mathfrak p\mid p\mathcal O_K$, the polynomial $X^5-p$ is Eisenstein.
The extension is totally ramified of degree five there.
Since $p\ne5$, this ramification is tame, and every nontrivial character has conductor exponent one.
There is no ramification away from $5p$ by the polynomial discriminant.
The local splitting at $5$ therefore gives \eqref{eq:quintic-exact-conductor}.
The conductor--discriminant formula yields
\[
\mathfrak d_{L_p/K}=(p\mathcal O_K)^4.
\]
As $|D_K|=125$, the tower formula gives
$|D_{L_p}|=125^5p^{16}=5^{15}p^{16}$.
\end{proof}

\subsection{An unconditional finite interval}

The two residue tests at $11$ and $31$ suffice for a short congruence argument.

\begin{lemma}\label{lem:quintic-small-primes}
Let $u$ be an odd positive integer with $u<1100$.
If
\[
u^4\equiv1\pmod{25},\qquad
u^2\equiv1\pmod{11},\qquad
u^6\equiv1\pmod{31},
\]
then $u\in\{1,243,749\}$.
None of these integers is prime.
\end{lemma}

\begin{proof}
The roots modulo $25$ are $1,7,18,24$, and the roots modulo $11$ are $1,-1$.
The Chinese remainder theorem, including oddness, gives exactly the eight classes
\[
u=r+550t,\qquad
r\in\{1,43,199,243,307,351,507,549\}.
\]
Since $0<u<1100$, one has $t\in\{0,1\}$.
The six roots of $X^6-1$ modulo $31$ are
\[
H=\{1,5,6,25,26,30\}.
\]
Each displayed residue is a root, and the degree bound over the field $\mathbb F_{31}$ proves exhaustiveness.
As $550\equiv23\pmod{31}$, Table~\ref{tab:quintic-small-primes} checks both values of $t$.
The only survivors are $1$, $243=3^5$, and $749=7\cdot107$.
\end{proof}

\begin{table}[ht]
\centering
\caption{The eight congruence classes below $1100$. The last column retains residues in $H$.}
\label{tab:quintic-small-primes}
\begin{tabular}{rrrr}
\toprule
$r$ & $r\bmod31$ & $(r+550)\bmod31$ & Allowed $u$\\
\midrule
1   & 1  & 24 & 1\\
43  & 12 & 4  & ---\\
199 & 13 & 5  & 749\\
243 & 26 & 18 & 243\\
307 & 28 & 20 & ---\\
351 & 10 & 2  & ---\\
507 & 11 & 3  & ---\\
549 & 22 & 14 & ---\\
\bottomrule
\end{tabular}
\end{table}

\begin{proposition}\label{prop:quintic-finite}
If $n\ge8$, $M\ge2$, and $|n!-M^5|\le n$, then the displacement is a prime $p\ge1100$ with $n/2<p\le n$.
\end{proposition}

\begin{proof}
Theorem~\ref{thm:quintic-elementary-range} gives $n\ge101$.
Lemma~\ref{lem:quintic-residue-conditions} then gives \eqref{eq:quintic-shift} and $p>50$.
Both $11$ and $31$ are smaller than $p$ and at most $n$.
Thus \eqref{eq:quintic-mod25} and \eqref{eq:quintic-residue} give the three congruences of Lemma~\ref{lem:quintic-small-primes}.
That lemma excludes $p<1100$.
\end{proof}

\begin{corollary}\label{cor:quintic-finite}
For every $8\le n\le1100$, every $m\ge2$, and every positive exponent $k$ divisible by $5$,
\[
|n!-m^k|>n.
\]
\end{corollary}

\begin{proof}
Write $m^k=(m^{k/5})^5$ and apply Proposition~\ref{prop:quintic-finite}.
A displacement prime would satisfy $1100\le p\le n\le1100$, forcing the composite value $p=1100$.
\end{proof}

\subsection{Unconditional analytic separation for large indices}\label{sec:quintic-analytic}

We now obtain a detecting prime without GRH, at the cost of an uncomputed effective threshold.

All ideals below are nonzero integral ideals of $\mathcal O_K$.
Write
\[
\kappa=\operatorname{Res}_{s=1}\zeta_K(s)>0.
\]
The prime above $5$ has norm $5$.

We use the following standard unconditional estimates without reproving them.
The first two follow from \cite[Theorems~1 and~5]{GarciaLee}; the convexity estimate is the specialization of \cite[Lemma~3.1]{Zaman} to degree four and $\delta=1/100$.

\begin{lemma}\label{lem:quintic-analytic-inputs}
For this fixed field $K$, with effective implied constants,
\begin{align}
A_K(t)&:=\#\{\mathfrak a:\Nm\mathfrak a\le t\}
=\kappa t+O_K(t^{3/5}),\label{eq:quintic-analytic-ideal-count}\\
\sum_{\Nm\mathfrak q\le t}\dfrac1{\Nm\mathfrak q}
&=\log\log t+M_K+O_K(1/\log t),\qquad t\ge2,\label{eq:quintic-analytic-mertens}
\end{align}
where $\mathfrak q$ runs over prime ideals.
If $\psi$ is a primitive nonprincipal finite-order Hecke character with conductor $\mathfrak f$, then
\begin{equation}\label{eq:quintic-analytic-convexity}
L(1/2+it,\psi)
\ll_K (\Nm\mathfrak f)^{51/200}(2+|t|)^{51/50}.
\end{equation}
The same upper bound holds uniformly for $1/2\le\operatorname{Re}s\le2$, with $t=\operatorname{Im}s$.
\end{lemma}

The cited convexity estimate covers $1/2\le\operatorname{Re}s\le1+\delta$, with $\delta=1/100$.
On $1+\delta\le\operatorname{Re}s\le2$, absolute convergence gives $|L(s,\psi)|\le\zeta_K(1+\delta)$, completing the stated strip bound.

In \eqref{eq:quintic-analytic-ideal-count}, it is enough to take $t\ge1$.
The fixed positive residue $\kappa$ and all constants used here are effective.
We also use $\pi_K(t)\ll_Kt/\log t$, which follows from the rational-prime Chebyshev bound and the fact that at most four prime ideals lie above any rational prime.

Define the nonnegative compactly supported weight
\[
w(t)=\begin{cases}(1-t)^2,&0\le t\le1,\\0,&t>1.\end{cases}
\]
Its Mellin transform on $\operatorname{Re}s>0$ is
\begin{equation}\label{eq:quintic-analytic-mellin}
\widehat w(s)=\int_0^1(1-t)^2t^{s-1}\,dt
=\dfrac2{s(s+1)(s+2)}.
\end{equation}
For ideals coprime to $5$, partial summation in \eqref{eq:quintic-analytic-ideal-count} gives
\begin{equation}\label{eq:quintic-analytic-B}
B(T):=\sum_{(\mathfrak a,5)=1}w(\Nm\mathfrak a/T)
=\dfrac c3T+O_K(T^{3/5}),\qquad T\ge1,
\end{equation}
where $c=(1-1/5)\kappa>0$.
Indeed, the unweighted count is $A_K(T)-A_K(T/5)=cT+O_K(T^{3/5})$, and $\int_0^1w(t)\,dt=1/3$.
In particular, $B(T)\ll_KT$ for $T\ge1$.

\begin{theorem}\label{thm:quintic-analytic-detector}
There is an effectively computable constant $P_0>5$ such that every rational prime $p\ge P_0$ admits a rational prime $q$ satisfying
\begin{equation}\label{eq:quintic-analytic-detector}
q\le p^{19/20},\qquad q\equiv1\pmod5,\qquad
p^{(q-1)/5}\not\equiv1\pmod q.
\end{equation}
\end{theorem}

\begin{proof}
Fix a rational prime $p\ne5$.
By Lemma~\ref{lem:quintic-fields}, the extension $L_p/K$ is cyclic of degree five and has a faithful Hecke character $\chi$ of order five.
Its nontrivial powers have the common conductor bounded in \eqref{eq:quintic-analytic-conductor}.
No residue condition on $p$ modulo $25$ is needed for this argument.

For $1\le j\le4$, let $\chi_j$ agree with $\chi^j$ on ideals coprime to $5p$ and vanish on all other ideals.
Its $L$-function is that of the associated primitive character with some Euler factors above $5p$ deleted.
There are at most five such prime ideals.
On $\operatorname{Re}s=1/2$, their additional factors are bounded uniformly in $p$.

For $X\ge1$, set
\[
S_j(X)=\sum_{\mathfrak a}\chi_j(\mathfrak a)w(\Nm\mathfrak a/X).
\]
Mellin inversion and \eqref{eq:quintic-analytic-mellin} give
\[
S_j(X)=\dfrac1{2\pi i}\int_{(2)}
L(s,\chi_j)X^s\dfrac2{s(s+1)(s+2)}\,ds.
\]
Shift the line to $\operatorname{Re}s=1/2$.
No pole is crossed, since $\chi^j$ is nonprincipal.
The strip bound in Lemma~\ref{lem:quintic-analytic-inputs} makes the horizontal integrals tend to zero.
Equations~\eqref{eq:quintic-analytic-convexity} and \eqref{eq:quintic-analytic-conductor} imply
\begin{equation}\label{eq:quintic-analytic-twisted}
S_j(X)\ll_Kp^{51/50}X^{1/2}.
\end{equation}
The integral converges absolutely on that line, because the integrand has decay $O((1+|t|)^{-99/50})$ apart from its constant factors.

Let
\[
T_p(X)=\sum_{\substack{(\mathfrak a,5p)=1\\\chi(\mathfrak a)=1}}
w(\Nm\mathfrak a/X).
\]
Character orthogonality yields
\[
T_p(X)=\dfrac15\left(
\sum_{(\mathfrak a,5p)=1}w(\Nm\mathfrak a/X)
+\sum_{j=1}^4S_j(X)\right).
\]
Removing ideals divisible by a prime above $p$ from \eqref{eq:quintic-analytic-B} costs $O_K(X/p)$, since those prime ideals have norm at least $p$.
It follows that
\begin{equation}\label{eq:quintic-analytic-orthogonality}
T_p(X)=\dfrac c{15}X+
O_K\bigl(X^{3/5}+X/p+p^{51/50}X^{1/2}\bigr).
\end{equation}

An ideal is $Y$-smooth if each of its prime ideal factors has norm at most $Y$.
Let
\[
\Psi(X,Y)=\sum_{\substack{(\mathfrak a,5)=1\\\mathfrak a\;Y\text{-smooth}}}
w(\Nm\mathfrak a/X).
\]
Assume $X\ge Y>5$.
Every nonsmooth ideal counted in $B(X)$ has a prime ideal factor $\mathfrak q$ with $Y<\Nm\mathfrak q\le X$.
A union bound therefore gives the exact inequality
\begin{equation}\label{eq:quintic-analytic-union}
\Psi(X,Y)\ge B(X)-
\sum_{Y<\Nm\mathfrak q\le X}B(X/\Nm\mathfrak q).
\end{equation}
Repeated counting on the right only increases the subtracted quantity.
By \eqref{eq:quintic-analytic-B},
\[
\sum_{Y<\Nm\mathfrak q\le X}B(X/\Nm\mathfrak q)
=\dfrac c3X\sum_{Y<\Nm\mathfrak q\le X}\dfrac1{\Nm\mathfrak q}
+O_K\left(X^{3/5}\sum_{Y<\Nm\mathfrak q\le X}
(\Nm\mathfrak q)^{-3/5}\right).
\]
Partial summation with $\pi_K(t)\ll_Kt/\log t$ gives
\[
\sum_{Y<\Nm\mathfrak q\le X}(\Nm\mathfrak q)^{-3/5}
\ll_K\dfrac{X^{2/5}}{\log Y}.
\]
Combining this with \eqref{eq:quintic-analytic-mertens} proves
\begin{equation}\label{eq:quintic-analytic-smooth-lower}
\Psi(X,Y)\ge\dfrac c3X
\left(1-\log\dfrac{\log X}{\log Y}\right)
-O_K\bigl(X/\log Y+X^{3/5}\bigr).
\end{equation}
No asymptotic theorem for smooth ideals is needed.

To make the character error smaller than the main term, we require $X/p^{51/25}\to\infty$.
To make the smooth-ideal lower bound exceed the character-kernel main term, we require $\log X/\log Y<e^{4/5}$.
Both conditions, together with $Y<p$, hold for
\[
X=p^{41/20},\qquad Y=p^{19/20}.
\]
Since $p^{51/50}X^{1/2}=Xp^{-1/200}$, equation~\eqref{eq:quintic-analytic-orthogonality} becomes
\[
\dfrac{T_p(X)}X=\dfrac c{15}
+O_K\bigl(p^{-41/50}+p^{-1}+p^{-1/200}\bigr).
\]
Suppose every prime ideal $\mathfrak q\nmid5p$ with $\Nm\mathfrak q\le Y$ has $\chi(\mathfrak q)=1$.
Every $Y$-smooth ideal coprime to $5$ is also coprime to $p$, because $Y<p$.
Its character is one, so $\Psi(X,Y)\le T_p(X)$.
But \eqref{eq:quintic-analytic-smooth-lower} gives
\[
\dfrac{\Psi(X,Y)}X\ge\dfrac c3\bigl(1-\log(41/19)\bigr)
-O_K\bigl(1/\log p+p^{-41/50}\bigr).
\]
The leading coefficient exceeds $c/15$, since
\[
\log(41/19)<\dfrac45.
\]
For example, the exponential series gives
$e^{4/5}>1+4/5+(4/5)^2/2+(4/5)^3/6=827/375>41/19$.
This is a contradiction for all sufficiently large $p$.
More explicitly, the false hypothesis would force
\[
\dfrac c3\left(\dfrac45-\log(41/19)\right)
\ll_K \dfrac1{\log p}+p^{-1/200}+p^{-1}+p^{-41/50}.
\]
All implied constants depend only on the fixed field $K$ and are effective.
The left side is a fixed positive number, and the right side tends effectively to zero.
Thus a detecting prime ideal $\mathfrak q\nmid5p$ of norm at most $Y$ exists once $p$ exceeds an effective constant.

Let $q$ be its underlying rational prime.
If $q\not\equiv1\pmod5$, fifth powering is a bijection on $\mathbb F_q^\times$.
The rational residue of $p$ would then already have a fifth root in $\mathbb F_q$, hence also in $\mathcal O_K/\mathfrak q$.
Its Kummer Frobenius would be trivial, contrary to $\chi(\mathfrak q)\ne1$.
Therefore $q\equiv1\pmod5$.
Such a prime splits completely in $K$, so $\Nm\mathfrak q=q$.
The nontrivial Kummer character says precisely that $p$ is not a fifth power modulo $q$.
This proves the theorem.
\end{proof}

\begin{theorem}\label{thm:quintic-eventual}
There is an effectively computable integer $N_5\ge1101$ such that
\[
|n!-m^k|>n
\]
for every integer $n\ge N_5$ and all integers $m,k\ge2$ with $5\mid k$.
The threshold is uniform in both $m$ and $k$.
\end{theorem}

\begin{proof}
First take $k=5$.
Choose an integer $N_5\ge\max\{1101,2P_0\}$, where $P_0$ is the effective constant in Theorem~\ref{thm:quintic-analytic-detector}.
If a counterexample existed with $n\ge N_5$, Lemma~\ref{lem:quintic-residue-conditions} would give a prime displacement $n/2<p\le n$ and the necessary congruence \eqref{eq:quintic-residue}.
Since $p>P_0$, Theorem~\ref{thm:quintic-analytic-detector} supplies a prime
\[
q\le p^{19/20}<p\le n,\qquad q\equiv1\pmod5,
\]
violating that congruence.
This proves the assertion for fifth powers.
For $5\mid k$, apply the fifth-power assertion to $m^k=(m^{k/5})^5$.
\end{proof}

No numerical value of $P_0$ or $N_5$ is obtained.
Combined with Corollary~\ref{cor:quintic-finite} and Theorem~\ref{thm:bounded-base}, the result restricts any remaining fifth-power counterexample to
\[
1100<n<N_5,\qquad P^+(m)>n+2.
\]
The elementary initial-interval theorem remains independent of this analytic argument.

\subsection{Completion under GRH}\label{sec:quintic-erh-completion}

We now use the fifth-power GRH hypothesis stated before Theorem~\ref{thm:main-quintic}.
The conditional input is a specialization of the explicit-formula argument of Wesolowski \cite[Section~3]{Wesolowski}.
The exact conductor in Lemma~\ref{lem:quintic-exact-conductor} is essential for the threshold below.

We first record the elementary bounds for the scalar constants.
For rational $v>1$, put $z=(v-1)/(v+1)$ and
\[
A_J(v)=2\sum_{j=0}^{J}\dfrac{z^{2j+1}}{2j+1},\qquad
B_J(v)=A_J(v)+\dfrac{2z^{2J+3}}{(2J+3)(1-z^2)}.
\]
The integrated geometric series gives $A_J(v)<\log v<B_J(v)$.
Writing $H_{20}=\sum_{j=1}^{20}1/j$, direct rational substitution gives
\begin{equation}\label{eq:quintic-scalar-bounds}
\begin{gathered}
A_3(2)>\dfrac{693}{1000},\qquad
B_5(5)<\dfrac{161}{100},\qquad
A_5(157/50)>\dfrac{1144}{1000},\\
H_{20}-B_{50}(20)-\dfrac1{40}>\dfrac{577}{1000},\\
10B_5(2)+B_2(275/256)<\dfrac{701}{100}.
\end{gathered}
\end{equation}
The classical bound $\pi>157/50$ gives $\log\pi>1.144$.
The trapezoidal inequality for $1/t$ gives
$\gamma>H_{20}-\log20-1/40$, where $\gamma$ is Euler's constant.
Consequently, \eqref{eq:quintic-scalar-bounds} implies
\begin{equation}\label{eq:quintic-numerical-inputs}
\begin{gathered}
\log5<1.61,\qquad \gamma>0.577,\qquad
\log2>0.693,\qquad \log\pi>1.144,\\
\log1100<7.01,\qquad \sqrt{1100}>33.
\end{gathered}
\end{equation}
For the last line, use $1100=2^{10}(275/256)$ and $33^2<1100$.
All finite decimals in these inequalities are exact rational endpoints.

\begin{lemma}\label{lem:quintic-erh}
Let $p\ge1100$ be prime with $p^4\equiv1\pmod{25}$.
Assume GRH for $\zeta_{L_p}$.
There is a rational prime $q<p$ with $q\equiv1\pmod5$ such that
\begin{equation}\label{eq:quintic-nonresidue}
p^{(q-1)/5}\not\equiv1\pmod q.
\end{equation}
\end{lemma}

\begin{proof}
Let $\chi$ be a faithful primitive Hecke character of $L_p/K$.
It has order five and conductor $p\mathcal O_K$ by Lemma~\ref{lem:quintic-exact-conductor}.
The factorization
\[
\zeta_{L_p}(s)=\zeta_K(s)\prod_{j=1}^{4}L(s,\chi^j)
\]
has entire nonprincipal factors.
Hence GRH for $\zeta_{L_p}$ implies the required zero assertions for $\zeta_K$ and $L(s,\chi)$ in the critical strip.

Suppose there is no detecting prime $q<p$ as in the statement.
Every prime ideal $\mathfrak q$ of $K$ with $\Nm\mathfrak q<p$ then has $\chi(\mathfrak q)=1$.
Indeed, for an underlying rational prime $q\ne5$ with $q\not\equiv1\pmod5$, fifth powering is bijective on $\mathbb F_q^\times$.
The rational residue of $p$ already has a fifth root in $\mathbb F_q$, hence in the residue field at $\mathfrak q$.
For $q\equiv1\pmod5$ the same conclusion follows from our supposition.
These give trivial Kummer Frobenius.
The prime above $5$ splits by Lemma~\ref{lem:quintic-exact-conductor}, and every conductor prime has norm at least $p$.

In the notation of \cite[Section~3]{Wesolowski}, put $\mathfrak m=p\mathcal O_K$ and $G=\operatorname{Cl}_{\mathfrak m}(K)$.
Take $H=G$ and $a=1/2$; the auxiliary character $\theta$ is then principal of conductor one.
The sums $S_H$ and $S_{\mathfrak m}$ there are weighted von Mangoldt sums over ideals coprime and noncoprime to $\mathfrak m$, respectively; $S_H$ contains the factor $1-\chi(\mathfrak a)$.
The terms $I_{1/2}$, $I_0$ and $I_-$ in its Lemma~3.4 collect the critical-zero, remaining-pole and trivial-zero contributions.
Let $\mathcal Z_K$ and $\mathcal Z_\chi$ denote the critical-strip zeros of $\zeta_K$ and $L(s,\chi)$, counted with multiplicity.
Put
\[
R=\sum_{\rho\in\mathcal Z_K}\dfrac1{|\rho+1/2|^2}
 +\sum_{\rho\in\mathcal Z_\chi}\dfrac1{|\rho+1/2|^2},
\qquad \ell=\log p.
\]
Thus $R$ is $R(1/2,\chi)$ in the cited paper.
For clarity, if
\[
D(s)=\dfrac{\zeta_K'}{\zeta_K}(s)-\dfrac{L'}L(s,\chi),
\]
then its remaining-pole term specializes to
\[
I_0(x)=\dfrac{\log x\,D(-1/2)+D'(-1/2)}{\sqrt x}-4.
\]
The constant $-4$ is the contribution $-1/a^2$ at $a=1/2$.
Retain the exact conductor term $\log(125^2p^4)$ in equation~(3-1) of \cite{Wesolowski}, which occurs in the proof of its Lemma~3.6.
For real $s>1$, one may discard the combined real logarithmic derivative, since
\[
\operatorname{Re}\left(\dfrac{\zeta_K'}{\zeta_K}(s)
 +\dfrac{L'}L(s,\chi)\right)
=-\sum_{\mathfrak a}\dfrac{\Lambda_K(\mathfrak a)
 (1+\operatorname{Re}\chi(\mathfrak a))}{(\Nm\mathfrak a)^s}\le0.
\]
Here $\Lambda_K(\mathfrak q^j)=\log\Nm\mathfrak q$ and is zero on other ideals.
Thus the cited equation, with $[K:\mathbb Q]=4$ and no real places, gives
\begin{align}
R&\le\dfrac12\left[\log(125^2p^4)+\dfrac{16}{3}
 +8\bigl(\psi(3/2)-\log(2\pi)\bigr)\right]\nonumber\\
 &=2\ell+c_0,\label{eq:quintic-zero-sum}
\end{align}
where $\psi=\Gamma'/\Gamma$ and
\[
c_0=3\log5+\dfrac{32}{3}-4\gamma-12\log2-4\log\pi.
\]
We used $\psi(3/2)=2-\gamma-2\log2$.
The rational bounds \eqref{eq:quintic-numerical-inputs} give
\[
c_0<3(1.61)+\dfrac{32}{3}-4(0.577)-12(0.693)-4(1.144)
=\dfrac{89}{300}<\dfrac3{10}.
\]

At $x=p$, the ideal sums $S_H$ and $S_{\mathfrak m}$ of \cite[Lemma~3.4]{Wesolowski} vanish.
The first uses character triviality on every ideal of norm below $p$; the second uses the conductor-prime norm bound.
There are no real places, so $I_-=0$.
The cited lemma and \cite[Lemmas~3.7 and~3.9]{Wesolowski} specialize to
\begin{gather*}
\dfrac{4p}{9}=I_{1/2}+I_0,\qquad
|I_{1/2}|\le\sqrt p\,R,\\
\operatorname{Re}I_0\le
\dfrac{((5/2)\ell+1)R+(25/6)\ell+40/9}{\sqrt p}-4.
\end{gather*}
Taking real parts and dividing by $\sqrt p$, we obtain
\begin{equation}\label{eq:quintic-explicit-comparison}
\dfrac49\sqrt p\le
R\left(1+\dfrac{(5/2)\ell+1}{p}\right)
-\dfrac4{\sqrt p}+\dfrac{(25/6)\ell+40/9}{p}.
\end{equation}
Discard the negative term, substitute $R\le2\log p+3/10$, and divide by $\sqrt p$.
This forces $4/9\le F(p)$, where
\[
F(x)=\dfrac{2\log x+3/10}{\sqrt x}
 +\dfrac{5(\log x)^2+(83/12)\log x+427/90}{x^{3/2}}.
\]
Both summands are strictly decreasing for $x\ge1100$.
The derivative of the first has the sign of $37/20-\log x<0$.
For the second, put $P(v)=5v^2+(83/12)v+427/90$ and note that
\[
P'(v)-\dfrac32P(v)
=-\dfrac{15}{2}v^2-\dfrac38v-\dfrac15<0\qquad(v>0).
\]
Finally, \eqref{eq:quintic-numerical-inputs} gives the strict rational comparison
\[
F(1100)<\dfrac{144458377}{326700000}
=\dfrac49-\dfrac{741623}{326700000}<\dfrac49.
\]
This contradicts $4/9\le F(p)$ for every $p\ge1100$ and proves the lemma.
\end{proof}

The preceding argument specializes the published explicit formulas and does not invoke the general numerical bound of \cite[Theorem~1.1]{Wesolowski} with a modified constant.
The conductor contribution in \eqref{eq:quintic-zero-sum} is $\log(125^2p^4)$; replacing it by the generic bound $2\log(125p^4)$ would lose the estimate used here.

\begin{theorem}\label{thm:quintic-erh}
Assume the fifth-power GRH hypothesis specified before Theorem~\ref{thm:main-quintic}.
For every integer $n\ge8$ and every integer $M\ge2$,
\[
|n!-M^5|>n.
\]
\end{theorem}

\begin{proof}
If the inequality failed, Proposition~\ref{prop:quintic-finite} would give \eqref{eq:quintic-shift} with a prime $p\ge1100$.
By \eqref{eq:quintic-mod25}, this prime belongs to the family covered by the fifth-power GRH hypothesis.
Lemma~\ref{lem:quintic-erh} supplies $q<p\le n$ with $q\equiv1\pmod5$ and \eqref{eq:quintic-nonresidue}.
This contradicts \eqref{eq:quintic-residue}.
\end{proof}

\begin{corollary}\label{cor:quintic-multiples}
Under the same GRH hypothesis, for every integer $n\ge8$ and all integers $m,k\ge2$ with $5\mid k$,
\[
|n!-m^k|>n.
\]
\end{corollary}

\begin{proof}
Write $k=5s$, where $s\ge1$ is an integer, and put $M=m^s\ge2$.
Then $m^k=M^5$, so Theorem~\ref{thm:quintic-erh} gives the assertion.
\end{proof}

\section{Remaining cases}\label{sec:remaining}

The preceding results do not establish Conjecture~\ref{conj:fsc} in full.
They give the following exact reduction.

\begin{corollary}\label{cor:final-reduction}
Conjecture~\ref{conj:fsc} is equivalent to excluding, for $n\ge8$, the two families
\begin{align}
n!&=a^2(b^2-1), &1\le a^2&\le n,\label{eq:final-squares}\\
m^\ell&=n!+\varepsilon p, &n/2&<p\le n,\label{eq:final-odd}
\end{align}
where $a,b\ge1$, $m\ge2$, $\ell\ge5$ is an odd prime, $p$ is prime and $\varepsilon\in\{-1,1\}$.
In \eqref{eq:final-squares}, one necessarily has $\gcd(a,b)=1$, $b$ odd, $\gcd(b,n!)=1$ and $P^+(b)>n+2$.
In \eqref{eq:final-odd}, one may additionally require $P^+(m)>n+2$.
Unconditionally, the subfamily $\ell=5$ may be restricted to $1100<n<N_5$, where $N_5$ is the effective uncomputed threshold of Theorem~\ref{thm:quintic-eventual}.
Under the GRH hypothesis of Section~\ref{sec:quintic-erh-completion}, the second family may instead be restricted to $\ell\ge7$.
\end{corollary}

\begin{proof}
Choose a prime divisor $\ell$ of the exponent of a counterexample and rewrite the same number as an $\ell$-th power.
For $\ell=2$, Theorem~\ref{thm:classification} gives \eqref{eq:final-squares}.
Since $a^2\le n$, every prime divisor of $a$ is at most $n$; applying Theorem~\ref{thm:bounded-base} to the base $ab$ therefore gives $P^+(b)>n+2$.
Theorem~\ref{thm:cubic-separation} excludes $\ell=3$.
For an odd prime $\ell\ge5$, Theorem~\ref{thm:odd-shift} gives \eqref{eq:final-odd}.
Theorem~\ref{thm:bounded-base} further gives $P^+(m)>n+2$ in this family.
Conversely, each displayed equation violates the conjectured inequality.
Corollary~\ref{cor:quintic-finite} and Theorem~\ref{thm:quintic-eventual} give the finite index restriction when $\ell=5$.
The conditional restriction follows from Theorem~\ref{thm:quintic-erh}.
\end{proof}

Thus the unresolved square case includes the Brocard--Ramanujan conjecture.
At odd exponents, the remaining equations impose simultaneous residue conditions at small primes.
The unconditional analytic argument excludes all sufficiently large fifth-power displacement primes, but its threshold has not been computed.
Completing the fifth-power family without GRH therefore requires an effective overlap with the finite interval proved here, or another argument covering the remaining indices.
In the global conditional argument, the sole conditional input remains the detecting prime supplied by Lemma~\ref{lem:quintic-erh}.

\bibliographystyle{amsplain}
\bibliography{references}

@article{BerndtGalway,
  author = {Berndt, B. C. and Galway, W. F.},
  title = {On the {Brocard--Ramanujan} {Diophantine} equation {$n!+1=m^2$}},
  journal = {Ramanujan J.},
  volume = {4},
  year = {2000},
  pages = {41--42},
  doi = {10.1023/A:1009873805276},
  note = {\url{https://doi.org/10.1023/A:1009873805276}}
}

@misc{Cagarrinho,
  author = {Cagarrinho, F.},
  title = {How to prove that a sufficiently large prime {$p$} can't be a cubic residue modulo every prime {$q<p$}?},
  year = {2025},
  note = {MathOverflow, question 492391; answer by Wojowu incorporating a suggestion of Will Sawin; \url{https://mathoverflow.net/questions/492391}},
  url = {https://mathoverflow.net/questions/492391}
}

@article{ErdosOblath,
  author = {Erd\H{o}s, P. and Obl\'ath, R.},
  title = {{\"U}ber diophantische {Gleichungen} der {Form} {$n!=x^p\pm y^p$} und {$n!\pm m!=x^p$}},
  journal = {Acta Litt. Sci. Szeged},
  volume = {8},
  year = {1937},
  pages = {241--255},
  url = {https://users.renyi.hu/~p_erdos/1937-09.pdf},
  note = {\url{https://users.renyi.hu/~p_erdos/1937-09.pdf}}
}

@article{GarciaLee,
  author = {Garcia, S. R. and Lee, E. S.},
  title = {Unconditional explicit {Mertens'} theorems for number fields and {Dedekind} zeta residue bounds},
  journal = {Ramanujan J.},
  volume = {57},
  year = {2022},
  pages = {1169--1191},
  doi = {10.1007/s11139-021-00435-6},
  note = {\url{https://doi.org/10.1007/s11139-021-00435-6}}
}

@book{HardyWright,
  author = {Hardy, G. H. and Wright, E. M.},
  title = {An Introduction to the Theory of Numbers},
  edition = {sixth},
  publisher = {Oxford University Press},
  address = {Oxford},
  year = {2008},
  note = {Revised by D. R. Heath-Brown and J. H. Silverman}
}

@book{IrelandRosen,
  author = {Ireland, K. and Rosen, M.},
  title = {A Classical Introduction to Modern Number Theory},
  edition = {second},
  series = {Graduate Texts in Mathematics},
  volume = {84},
  publisher = {Springer},
  address = {New York},
  year = {1990},
  doi = {10.1007/978-1-4757-2103-4},
  note = {\url{https://doi.org/10.1007/978-1-4757-2103-4}}
}

@misc{Mathphile,
  author = {{Mathphile}},
  title = {For any {$k>1$}, if {$n!+k$} is a square then will {$n\le k$} always be true?},
  year = {2020},
  note = {Mathematics Stack Exchange, question 3499670; answers by mathworker21 and WE Tutorial School; \url{https://math.stackexchange.com/questions/3499670}},
  url = {https://math.stackexchange.com/questions/3499670}
}

@misc{Milne,
  author = {Milne, J. S.},
  title = {{Class Field Theory}},
  year = {2020},
  note = {Version 4.03, online course notes; \url{https://www.jmilne.org/math/CourseNotes/CFT.pdf}},
  url = {https://www.jmilne.org/math/CourseNotes/CFT.pdf}
}

@article{Neumaier,
  author = {Neumaier, A.},
  title = {Bounding basis reduction properties},
  journal = {Des. Codes Cryptogr.},
  volume = {84},
  year = {2017},
  pages = {237--259},
  doi = {10.1007/s10623-016-0273-9},
  note = {\url{https://doi.org/10.1007/s10623-016-0273-9}}
}

@misc{Orr,
  author = {Orr, D.},
  title = {Sequences {A240939} and {A240940}: least nonnegative shifts from a factorial to a perfect power},
  year = {2014},
  note = {The On-Line Encyclopedia of Integer Sequences; companion sequence \url{https://oeis.org/A240940}; \url{https://oeis.org/A240939}},
  url = {https://oeis.org/A240939}
}

@article{PollackShapiro,
  author = {Pollack, R. M. and Shapiro, H. N.},
  title = {The next to last case of a factorial diophantine equation},
  journal = {Comm. Pure Appl. Math.},
  volume = {26},
  year = {1973},
  pages = {313--325},
  doi = {10.1002/cpa.3160260303},
  note = {\url{https://doi.org/10.1002/cpa.3160260303}}
}

@article{RosserSchoenfeld,
  author = {Rosser, J. B. and Schoenfeld, L.},
  title = {Approximate formulas for some functions of prime numbers},
  journal = {Illinois J. Math.},
  volume = {6},
  year = {1962},
  pages = {64--94},
  doi = {10.1215/ijm/1255631807},
  note = {\url{https://doi.org/10.1215/ijm/1255631807}}
}

@article{Zaman,
  author = {Zaman, A.},
  title = {On the least prime ideal and {Siegel} zeros},
  journal = {Int. J. Number Theory},
  volume = {12},
  year = {2016},
  pages = {2201--2229},
  doi = {10.1142/S1793042116501335},
  note = {\url{https://doi.org/10.1142/S1793042116501335}}
}

@article{Wesolowski,
  author = {Wesolowski, B.},
  title = {Generating subgroups of ray class groups with small prime ideals},
  journal = {The Open Book Series},
  volume = {2},
  number = {1},
  year = {2019},
  pages = {461--478},
  doi = {10.2140/obs.2019.2.461},
  note = {\url{https://doi.org/10.2140/obs.2019.2.461}}
}

\end{document}